\documentclass[10pt,electronic]{amsart}        

\usepackage{baskervald}
\usepackage[swedish, english]{babel}
\usepackage[T1]{fontenc}
\usepackage[latin1]{inputenc}
\usepackage{amsmath}
\usepackage{xcolor}
\usepackage{enumerate}
\usepackage{hyperref}
\usepackage[noabbrev,capitalize]{cleveref}
\hypersetup{
	colorlinks,
	linkcolor={red!50!black},
	citecolor={blue!50!black},
	urlcolor={blue!80!black}
}
\usepackage{comment}

\usepackage{amsthm}
\usepackage{amssymb}
\usepackage{mathrsfs}
\usepackage{stmaryrd}
\usepackage{mathtools}
\usepackage[shortlabels]{enumitem}
\usepackage{tikz-cd}

\usepackage{euscript}
\renewcommand{\mathcal}{\EuScript}

\usepackage[hmargin=3.2cm,vmargin=3.5cm]{geometry}

\author{Samir Canning}
\author{Dan Petersen}
\author{Olivier Ta\"ibi}
\title{The low degree cohomology of compactifications of $A_g$}

\usepackage[disable]{todonotes}

\renewcommand{\tilde}{\widetilde}
\newcommand{\Q}{\mathbf Q}
\newcommand{\Z}{\mathbf Z}
\newcommand{\R}{\mathbf R}
\newcommand{\C}{\mathbf C}
\newcommand{\A}{\mathbf A}
\newcommand{\V}{\mathbf V}
\newcommand{\G}{\mathbf G}
\newcommand{\Mbar}{\overline{M}}
\newcommand{\Abar}{\overline{A}}
\newcommand{\Xbar}{\overline{X}}

\newcommand{\IH}{\mathit{I\!H}}
\newcommand{\IC}{\mathrm{IC}}
\newcommand{\Sat}{\mathit{Sat}}
\newcommand{\Std}{\mathrm{Std}}
\newcommand{\std}{\mathrm{std}}
\DeclareMathOperator{\Spec}{Spec}
\DeclareMathOperator{\Sym}{Sym}
\DeclareMathOperator{\gr}{gr}
\DeclareMathOperator{\Res}{Res}
\DeclareMathOperator{\Hom}{Hom}
\newcommand{\rt}{\mathrm{rt}}
\newcommand{\aff}{\mathrm{aff}}
\DeclareMathOperator{\GL}{GL}
\DeclareMathOperator{\PGL}{PGL}
\DeclareMathOperator{\SL}{SL}
\DeclareMathOperator{\Sp}{Sp}
\DeclareMathOperator{\SO}{SO}
\DeclareMathOperator{\GSp}{GSp}
\DeclareMathOperator{\Spin}{Spin}
\DeclareMathOperator{\GSpin}{GSpin}
\newcommand{\spin}{\mathrm{spin}}
\DeclareMathOperator{\Frob}{Frob}
\newcommand{\cpsc}{c_{p,\mathrm{sc}}}
\newcommand{\Qp}{\mathbf{Q}_p}
\newcommand{\Qell}{\mathbf{Q}_\ell}
\newcommand{\Qellbar}{\overline{\mathbf{Q}}_\ell}
\newcommand{\Tcal}{\mathcal{T}}
\newcommand{\tPsidut}{\tilde{\Psi}_{\mathrm{disc}}^{\mathrm{unr}, \tau}}
\newcommand{\Zhat}{\widehat{\Z}}
\newcommand{\gfrak}{\mathfrak{g}}
\newcommand{\kfrak}{\mathfrak{k}}
\newcommand{\pfrak}{\mathfrak{p}}

\newcommand{\ol}[1]{\overline{#1}}

\newtheorem{thm}{Theorem}[section]
\newtheorem{cor}[thm]{Corollary}
\newtheorem{lem}[thm]{Lemma}
\newtheorem{prop}[thm]{Proposition}

\theoremstyle{definition}
\newtheorem{defn}[thm]{Definition}
\theoremstyle{remark}
\newtheorem{rem}[thm]{Remark}

\numberwithin{equation}{section}

\begin{document}

\begin{abstract}
    We compute the low degree $\ell$-adic intersection cohomology of symplectic local systems on $A_g^{\Sat}$, the Satake compactification of the moduli space $A_g$ of principally polarized abelian varieties. We prove that only a small finite list of irreducible Galois representations can appear in the low degree cohomology of any nonsingular toroidal compactification of $A_g$ or $X_{g,s}$, the $s$-fold fiber product of the universal abelian variety. We give several applications, including to spaces of holomorphic forms on toroidal compactifications and to the cohomology of the interior. In particular, we give a complete characterization of when $H^\bullet(X_{g,s},\Qell)$ and $H^\bullet(\Xbar_{g,s},\Qell)$ are Tate, which is independent of the choice of toroidal compactification.
\end{abstract}


 \maketitle

\section{Introduction}

\subsection{Motivic predictions}
There is a web of conjectures, falling under the umbrella of the Langlands program and the conjectures of Fontaine--Mazur, making predictions about the cohomology of varieties, or more generally Deligne--Mumford stacks.
Namely, one expects there to be natural bijections between the isomorphism classes of the following objects:
\begin{itemize}
\item (certain) automorphic representations for $\GL_{n,\Q}$, for varying integers $n$,
\item (certain) representations of $\mathrm{Gal}(\overline \Q/\Q)$ over $\Qellbar$, for any prime number $\ell$,
\item pure motives\footnote{More precisely motives with coefficients in $\overline{\Q}$, i.e.\ motives with an action of a number field $E \subset \overline{\Q}$ that is allowed to grow.} over $\Q$.
\end{itemize}
It is a delicate matter to describe precisely the automorphic representations or Galois representations that should arise from geometry, but the geometry should impose significant restrictions on the automorphic side.
For example, if $X$ is a smooth proper Deligne--Mumford stack over $\mathrm{Spec}(\Z)$, the Galois representation (or ``motive'') given by the $\ell$-adic cohomology $H^\bullet(X,\Q_\ell)$ will be unramified at all primes, which is a very rare property. An everywhere unramified Galois representation should conjecturally correspond to an automorphic representation that is \emph{spherical} at all primes, i.e.\ a representation of \emph{level one}, or \emph{conductor one}.

Recent work of Chenevier, Renard, Lannes, and Ta\"ibi \cite{chenevierrenard,chenevierlannes,taibi,cheneviertaibi} has shown that level one automorphic representations of small ``motivic weight'' can be completely classified, meaning that the Galois representations that should contribute to the low-degree cohomology of smooth proper Deligne--Mumford stacks over $\Spec \Z$ are classified as well.
Up to twisting by $|\det|^w$ for some $w \in \tfrac{1}{2} \Z$ we can reduce to automorphic representations for $\PGL_{n,\Q}$.
Chenevier--Lannes \cite[Theorem F]{chenevierlannes} obtain a complete list of all algebraic cuspidal automorphic representations of level one whose motivic weight is at most $22$; in particular, apart from the trivial representation for $\PGL_{1,\Q}$ there is \emph{no} level one algebraic cuspidal automorphic representation for $\PGL_{n,\Q}$ of motivic weight less than $11$, for any $n$.
For odd motivic weights this was already known, by an argument of Mestre \cite[III, Remarque 1]{mestre}.
See also \cite{cheneviertaibi} for a simpler proof of \cite[Theorem F]{chenevierlannes}, and partial results in motivic weights 23 and 24.

\begin{thm}[Chenevier--Lannes]\label{thm:ChenevierLannes}
    There are $11$ algebraic cuspidal automorphic representations for $\PGL_{n,\Q}$ of level one of motivic weight at most $22$:
    \[
    1,\ \Delta_{11},\ \Delta_{15},\ \Delta_{17},\ \Delta_{19},\ \Delta_{19,7},\ \Delta_{21},\ \Delta_{21,5},\ \Delta_{21,9},\ \Delta_{21,13},\ \Sym^2 \Delta_{11}.
    \]
\end{thm}
Each automorphic representation in Theorem \ref{thm:ChenevierLannes} has an associated $\ell$-adic Galois representation. 
To the automorphic representation $\Delta_w$ is associated the two-dimensional Galois representation $S_{\ell}\langle w+1\rangle$ attached to weight $w+1$ cusp forms for $\SL_2(\Z)$; to the automorphic representation $\Delta_{w_1,w_2}$ is similarly associated a $4$-dimensional Galois representation attached to vector valued Siegel cusp forms for $\Sp_4(\Z)$ of stable type (it factors through $\GSp_4(\Qell)$), and to $\Sym^2\Delta_{11}$ is associated $\Sym^2 S_{\ell}\langle 12 \rangle$.

Given a smooth proper Deligne--Mumford stack $X$ over $\mathrm{Spec}(\Z)$ and $k \leq 22$, it is therefore a natural problem to verify that the only Galois representations appearing in $H^k(X,\Q_{\ell})$ are Tate twists of the ones associated to the automorphic representations in Theorem \ref{thm:ChenevierLannes}.
Smooth and proper stacks over $\mathrm{Spec}(\Z)$ are rare objects. The ones which first spring to mind (e.g.~flag varieties) have no interesting Galois representations in their cohomology at all.
A highly nontrivial example is $\Mbar_{g,n}$, the moduli space of stable curves of genus $g$ with $n$ markings, which was shown to be smooth and proper over $\mathrm{Spec}(\Z)$ by Deligne and Mumford \cite{dm69}.
Many of the predictions for $H^\bullet(\Mbar_{g,n})$ arising from Theorem \ref{thm:ChenevierLannes} have been verified in work of Arbarello--Cornalba \cite{arbarellocornalba-calculating}, Bergstr\"om--Faber--Payne \cite{bergstromfaberpayne}, and Canning--Larson--Payne \cite{CanningLarsonPayne11,CanningLarsonPayneMotivic}.
The following theorem is an amalgamation of these results.

\begin{thm} \label{thm:Mbar} 
For $k\leq 15$, the semisimplification of the Galois representation $H^k(\Mbar_{g,n},\Q_{\ell})$ is a direct sum of Tate twists of $1$, $S_{\ell}\langle 12\rangle$, and $S_{\ell}\langle 16\rangle$.
\end{thm}

The next natural cases to study are moduli spaces of abelian varieties.  Like $M_{g,n}$, the moduli stack $A_g$ of principally polarized abelian varieties is smooth but not proper over the integers; unlike $M_{g,n}$, it does not come with a God-given and most natural compactification, but rather a whole zoo of different compactifications with various advantages and disadvantages. The cohomology of $A_g$ with twisted coefficients has been extensively studied when $g$ is small \cite{bfg11,BF}, using point counting over finite fields, and the results provide remarkable consistency checks with the predictions arising from \cref{thm:ChenevierLannes}. For further computations of the cohomology of $A_g$ and its various compactifications, see e.g.\ \cite{hainA3,HT-3,HT-4,GH,HulekTommasiTaibi,w0,w0-2,borelstablereal1,CL1,CL,GHT}.

\subsection{The Satake compactification}
One natural compactification of $A_g$ is the \emph{Satake compactification} $A_g^\Sat$. It is a proper Deligne--Mumford stack over $\Z$ containing $A_g$ as a dense open substack, but it is highly singular. It may therefore be more natural to consider its \emph{intersection cohomology} groups $\IH^\bullet(A_g^\Sat,\Q_\ell)$. Given that the raison d'\^etre of intersection cohomology is that it should treat singular spaces as if they were smooth, it is perhaps not surprising that $\IH^\bullet(A_g^\Sat,\Q_\ell)$ is similarly everywhere unramified (see \cite[Proposition 4.5.2]{taibi_ecAn} or Lemma \ref{hansen lemma} below). Abstractly, one should thus expect similarly that $\IH^k(A_g^\Sat,\Q_\ell)$, for $k \leq 22$, decomposes into a direct sum of Tate twists of the Galois representations coming from Theorem \ref{thm:ChenevierLannes}.

In fact, there is a more direct link between intersection cohomology of the Satake compactification and automorphic representations, using which much more than the above expectation can be proved. In particular, it is striking that we can go beyond weight $22$ (in principle \emph{much} further, in particular with constant coefficients), even without knowing an analogue of \cref{thm:ChenevierLannes} in weight at least $23$.
We comment more on this in \cref{rem:77} and in \cref{rem:higher_wt}. Let us for simplicity first state the special case of cohomology with constant coefficients.

\begin{thm}\label{thm:lowdegintcoh}
We can explicitly compute $\IH^k(A_g^\Sat, \Qell)$ for all $g \geq 1$ and all $k \leq 23$. All cohomology is pure of Tate type, with the sole exception of $(g,k)=(7,22)$, in which case
\[ \IH^{22}(A_7^\Sat, \Qell) \simeq \Qell(-11)^{\oplus 10} \oplus \Sym^2 S_\ell \langle 12 \rangle. \]
In particular, $\IH^k(A_g^\Sat,\Q_\ell)$ vanishes for all odd $k \leq 23$ and all $g$.
\end{thm}

For $\lambda = (\lambda_1 \geq \lambda_2 \geq \dots \geq \lambda_g\geq 0)$ a dominant weight of $\Sp_{2g}$, we associate a lisse $\ell$-adic sheaf $\V_\lambda$ on $A_g$ of weight $|\lambda|$, and we may consider the intersection cohomology group $\IH^k(A_g^\Sat,\V_\lambda)$, which is pure of weight $k+|\lambda|$. For $k + |\lambda| \leq 22$, we expect it to decompose into summands given by Galois representations associated to the automorphic representations in Theorem \ref{thm:ChenevierLannes}. 

Let $Y_g$ be the Lagrangian Grassmannian of $g$-dimensional subspaces of a $2g$-dimensional space. Over the complex numbers, $Y_g$ is the compact dual of $A_g$.
The cohomology ring of $Y_g$ is well-understood: $H^\bullet(Y_g)$ is pure Tate and is isomorphic to the tautological ring of any nonsingular toroidal compactification $\Abar_g$ (see Section \ref{sec:compact} for definitions), which was calculated by van der Geer \cite{vdg}. In particular, its Poincar\'e polynomial is $\sum_{i\geq 0} b_{k} t^k = \prod_{j=1}^g(1+t^{2j})$.

\begin{thm}[Theorem \ref{thm IH tables}] \label{thm IH}
 For all but finitely many pairs $(g,\lambda)$, there exists an isomorphism 
	\begin{equation}\label{isom in thm IH} \IH^k(A_g^\Sat,\V_\lambda) \cong \begin{cases}
		H^k(Y_g,\Q_\ell) & \lambda = 0 \\
		0 & \lambda \neq 0
	\end{cases}\qquad \text{for all $k+|\lambda| \leq 23$.}\end{equation}
 We tabulate all exceptional cases $(g,\lambda)$ where \eqref{isom in thm IH} fails to hold, and in these cases we compute the semisimplification of $\IH^k(A_g^\Sat,\V_\lambda)$ for  $k+|\lambda| \leq 23$.
 All summands are Tate twists of the Galois representations associated to the automorphic representations in Theorem \ref{thm:ChenevierLannes} along with $\Delta_{23}^{(i)}$ ($i=1,2$, corresponding to the two eigenforms in $S_{24}(\SL_2(\Z))$), $\Delta_{23,7}$, and $\Delta_{23,9}$ in weight 23.
\end{thm}

\begin{rem}The first sentence of \cref{thm IH} is not new: since \eqref{isom in thm IH} is vacuous for $|\lambda| > 24$, it suffices to prove that \eqref{isom in thm IH} holds for all sufficiently large $g$ to see that there are only finitely many exceptions. Now by the  Zucker conjecture (theorem of \cite{looijengazucker,saperstern}), $\IH^\bullet(A_g^\Sat,\R)$ is isomorphic to the transcendentally defined $L^2$-cohomology of $A_g$, and the $L^2$-cohomology of $A_g$ was studied in detail in Borel's work on stable real cohomology of arithmetic groups \cite{borelstablereal1}. Indeed, Borel's results (specialized to symplectic groups) prove homological stability for $H^\bullet(A_g,\R)$ by first proving homological stability for $L^2$-cohomology, and then showing that $L^2$-cohomology and ordinary cohomology agree in a range; from this one may deduce \eqref{isom in thm IH} for sufficiently large $g$. We remark also that Borel's results are valid more generally for any congruence subgroup of $\Sp_{2g}(\Z)$.
\end{rem}

\begin{rem}\label{rem:77}
Let us explain why the methods used in this paper allow one to go beyond motivic weight 22. 
The point is that the intersection cohomology of the Satake compactification can be decomposed into summands indexed by \emph{Arthur--Langlands parameters}, each of which is described in terms of level 1 self-dual cuspidal automorphic representations of some $\PGL_{n,\Q}$.
But if an Arthur--Langlands parameter $\psi$ is built out of level $1$ cuspidal automorphic representations $\pi_0,\dots,\pi_r$, then in general the corresponding summand of intersection cohomology is \emph{not} built directly from the Galois representations attached to $\pi_0,\dots,\pi_r$, but rather from lifts to $\GSpin$ groups of these standard Galois representations (see \S 1.3.4--6 of \cite{taibi_ecAn}).
Hence the recipe for passing from an Arthur--Langlands parameter to a corresponding Galois representation appearing in the intersection cohomology is more subtle; in particular, the motivic weights of the constituents do not directly determine the weight of the summand of cohomology.  

The nontrivial constituents of an Arthur--Langlands parameter are always \emph{regular} level 1 cuspidal automorphic representations.
By contrast, Galois representations appearing in $\IH^\bullet(A_g^\Sat,\V_\lambda)$ are typically not associated to regular automorphic representations.
Moreover, one can bound from above the dimensions appearing in any fixed weight.
For example, it follows from \cref{lem:k_psi_i_ineq} that if $k+|\lambda| \leq 77$, then the Arthur--Langlands parameters contributing nontrivially to $\IH^k(A_g^\Sat,\V_\lambda)$ will only involve regular level 1 cuspidal automorphic representations for $\PGL_{n,\Q}$ for $n \leq 23$.
Enumerating regular level 1 cuspidal automorphic representations is a significantly easier task than enumerating \emph{all} of them, and can be done by computer up to rather high weights.  

The reason \cref{thm IH} can be formulated in terms of (standard) Galois representations associated to level $1$ cuspidal automorphic representations of $\PGL_n$ is that only $\pi_i$'s of small dimension occur, we have exceptional isomorphisms $\Spin_3 \simeq \SL_2$, $\Spin_4 \simeq \SL_2 \times \SL_2$ and $\Spin_5 \simeq \Sp_4$, \emph{and} we know the corresponding lifts on the automorphic side.
This breaks down in higher dimensions (but see Remark \ref{rem:higher_wt} for an extra case), and thus we cannot formulate the analogue of \cref{thm IH} in the same way in higher degree, even though we can write a precise list of (half-)spin Galois representations which can occur (the representations $\sigma_{\psi_0,\iota}^{\spin}$ and $\sigma_{\psi_i,\iota}^{\spin,\epsilon}$ of \cite[\S 1.3.6]{taibi_ecAn}).

Also note that not all known level one cuspidal automorphic representations for $\GL_{n,\Q}$ of motivic weight $23$ (see \cite[Theorem 3 and Theorem 4]{cheneviertaibi}) contribute to some $\IH^k(A_g^\Sat, \V_\lambda)$ for $k + |\lambda| \leq 23$.
A similar phenomenon occurs in degree $24$ as well: compare \cref{rem:higher_wt} and \cite[Theorem 5]{cheneviertaibi}.
\end{rem}

The bulk of the work in proving \cref{thm IH} is in prior work of the third author \cite{taibi_ecAn}. 

\subsection{Toroidal compactifications}
Beyond the Satake compactification, one may consider \emph{toroidal compactifications} of $A_g$, as constructed by Ash, Mumford, Rapoport, and Tai \cite{AMRT} over $\C$. The toroidal compactifications are normal crossings compactifications, and any toroidal compactification admits a resolution of singularities that is also a toroidal compactification.

Faltings and Chai extended the work in \cite{AMRT}, giving smooth and proper compactifications over $\Z$ \cite{faltingschai}. Moreover, let $X_g\rightarrow A_g$ be the universal principally polarized abelian $g$-fold, and let $X_{g,s}\rightarrow A_g$ be its $s$-fold fiber product. For all $g$ and $s$, Faltings and Chai constructed toroidal compactifications of $X_{g,s}$ that are smooth and proper over the integers. The case $s=0$ recovers the toroidal compactifications of $A_g$.

A toroidal compactification $\Abar^\Sigma_{g}$ depends on the choice of an admissible decomposition $\Sigma$ of the rational closure of the cone of positive definite quadratic forms on $\R^g$. Toroidal compactifications of $X_{g,s}$ also depend on admissible cone decompositions of a related cone (see Section \ref{sec:compact} for definitions). Certain aspects of the geometry are reflected in the combinatorics of the cone decomposition. In particular, the compactification is smooth as a stack over $\mathrm{Spec}(\Z)$ if the cone decomposition is simplicial.

\begin{thm}\label{thm:main}
    For $k\leq 23$ and for any nonsingular toroidal compactification $\Xbar_{g,s}$ of $X_{g,s}$, the semi-simplification of the Galois representation $H^k(\Xbar_{g,s},\Qellbar)$ is isomorphic to a direct sum of Tate twists of irreducible constituents of the Galois representations associated to the automorphic representations in Theorem \ref{thm IH}.
\end{thm}

Some of the most commonly studied toroidal compactifications of $A_g$ correspond to the perfect cone, second Voronoi, and central cone decompositions. None of these decompositions is simplicial for $g\geq 5$, and thus the corresponding compactifications are singular. Nevertheless, any admissible cone decomposition can be refined to a simplicial one, and the analogous result holds for cone decompositions giving rise to compactifications of $X_{g,s}$. The refinement induces a resolution of singularities on the level of the moduli stacks. Considering weights, we obtain the following corollary of Theorem \ref{thm:main}.
\begin{cor}\label{cor:purepart}
        For $k\leq 23$ and for any toroidal compactification $\Xbar_{g,s}$ of $X_{g,s}$, the Galois representation $\gr^W_k H^k(\Xbar_{g,s},\Q_{\ell})$ is a direct sum of Tate twists of the Galois representations associated to the automorphic representations in Theorem \ref{thm IH}.
\end{cor}

\begin{table}[h!]
\centering
\begin{tabular}{|c|c|c|c|c|c|c|c|c|c|} 
 \hline
 $g$ & $1$ & $2$ & $3$ & $4$ & $5$ & $6$ & $\geq 7$ \\
 \hline
 $c(g)$ & $10$ & $7$ & $6$ & $3$ & $2$ & $1$ & $0$  \\ 
\hline
\end{tabular}
\vspace{.1in}
\caption{}
\label{nontatetable}
\end{table}

It is natural to ask if any of the Galois representations associated to the non-trivial automorphic representations in Theorem \ref{thm:ChenevierLannes} actually appear in the cohomology of $\Xbar_{g,s}$. Let $c(g)$ be the function defined in Table \ref{nontatetable}.
\begin{thm}\label{thm:tateornot}
    Let $g\geq 1$ and $\Xbar_{g,s}$ be a nonsingular toroidal compactification of $X_{g,s}$. Then $H^{\bullet}(\Xbar_{g,s},\Q_{\ell})$ is not Tate if and only if $s\geq c(g)$. 
\end{thm}
The point count $\#\Xbar_{g,s}(\mathbb{F}_q)$ for any nonsingular toroidal compactification is a polynomial in $q$ if and only if the cohomology is Tate, so Theorem \ref{thm:tateornot} determines when the point counting function is a polynomial.
Note that the third author has proven that $\#A_g(\mathbb{F}_q)$ is a polynomial in $q$ when $g\leq 6$, computed the polynomials, and shown that $\#A_7(\mathbb{F}_q)$ is not polynomial in $q$ \cite[Theorem 1]{taibi_ecAn}.
This implies that $H^\bullet(A_7,\Q_{\ell})$ is not Tate.
We extend this latter result to higher genus and to fiber products of the universal abelian variety, and we prove a converse statement.
\begin{thm}\label{thm:interiornontate}
    Let $g\geq 1$. The Galois representation $H^\bullet(X_{g,s},\Q_{\ell})$ is not Tate if and only if $s\geq c(g)$.
\end{thm}
Theorems \ref{thm:tateornot} and Theorem \ref{thm:interiornontate} should be compared with analogous results on moduli spaces of curves. For example, $\#M_g(\mathbb{F}_q)$ is polynomial in $q$ if and only if $g\leq 8$, and $\#\Mbar_{g,n}(\mathbb{F}_q)$ is not polynomial for $g\geq 16$ \cite{CLPWpoly,CLPWFA}.

Theorems \ref{thm:main}, \ref{thm:tateornot}, \ref{thm:interiornontate} and Corollary \ref{cor:purepart} concern the $\ell$-adic cohomology as a Galois representation. Using Hodge--Tate comparison isomorphisms, we can also put strong restrictions on the $(p,q)$ with $p+q\leq 22$ such that $H^{p,q}(\Xbar_{g,s})\subset H^{p+q}(\Xbar_{g,s},\C)$ is nonzero. To state the result, we first introduce some notation. For each automorphic representation $\Delta$ in Theorem \ref{thm:ChenevierLannes}, there is an associated pure real Hodge structure $\mathrm{Hdg}(\Delta)$. We have $\mathrm{Hdg}(\Delta_w)^{p,q}=0$ for $(p,q)\notin \{(w,0),(0,w)\}$. We have $\mathrm{Hdg}(\Delta_{w_1,w_2})^{p,q}=0$ for $(p,q)\notin\{(w_1,0),(\tfrac{w_1+w_2}{2},\tfrac{w_1-w_2}{2}),(\tfrac{w_1-w_2}{2},\tfrac{w_1+w_2}{2}),(0,w_1)$\}.

\begin{thm}\label{thm:hodge}
    For any $k\leq 23$ and for any nonsingular toroidal compactification $\Xbar_{g,s}$, the Hodge structure $H^k(\Xbar_{g,s},\R)$ is a sum of effective Tate twists of the Hodge structures $\mathrm{Hdg}(\Delta)$ for $\Delta$ in the list of automorphic representations from Theorem \ref{thm IH}.
\end{thm}
We obtain stronger (non)vanishing results for the spaces of holomorphic forms $H^{p,0}(\Xbar_{g,s})$.
\begin{thm}\label{thm:holomorphic}
Let $\Xbar_{g,s}$ be a nonsingular toroidal compactification of $X_{g,s}$. If $g\geq 3$, then $H^{k,0}(\Xbar_{g,s})=0$ for all $s\geq 0$ and $0<k\leq 21$ or $k=23$.
\begin{enumerate}
    \item If $g = 3$, then $H^{22,0}(\Xbar_{g,s})\neq 0$ if and only if $s \geq 8$.
    \item If $4 \leq g \leq 7$, then $H^{22,0}(\Xbar_{g,s})\neq 0$ if and only if $s \geq c(g)$. 
    \item If $g \geq 8$, then $H^{22,0}(\Xbar_{g,s})= 0$.
\end{enumerate}
\end{thm}
 Dimension formulas for spaces of holomorphic forms $H^{k,0}(\Xbar_{g,s})$ for arbitrary $k$ and $s$ when $g=1,2$ can be extracted from the literature. See Remark \ref{rem:g12}.

\subsection{Plan of the paper}
In Section \ref{sec:Satake}, we discuss the intersection cohomology of the Satake compactification in terms of automorphic representation theory. In particular, we prove Theorems \ref{thm:lowdegintcoh} and \ref{thm IH}. We include Tables \ref{tab:rest}, \ref{tab:C}, and \ref{tab:DEF}, which contain the full details of the exceptional cases in Theorem \ref{thm IH}. In Section \ref{sec:compact}, we recall the toroidal compactifications of $A_g$ and $X_{g,s}$ and their stratifications according to cones in the corresponding cone decompositions. Then, in Section \ref{sec:stratacoh}, we explain how to compute the cohomology of the strata as a Galois representation in terms of intersection cohomology groups with twisted coefficients of the Satake compactification, which allows us to prove Theorem \ref{thm:main} and hence Corollary \ref{cor:purepart}. In Section \ref{sec:nontate}, we apply the results of Sections \ref{sec:Satake} and \ref{sec:stratacoh} to prove Theorems \ref{thm:tateornot}.  Section \ref{sec:Hodge} contains the proofs of Theorems \ref{thm:interiornontate}, \ref{thm:hodge}, and \ref{thm:holomorphic}. Finally, in Section \ref{sec:L2}, we sketch an alternate proof using $L^2$ cohomology of a weaker form of our main results. 

\subsection*{Acknowledgments} We are grateful to Ga\"etan Chenevier, Jeremy Feusi, Sam Grushevsky, Sophie Morel, Johannes Schmitt, and Pim Spelier for helpful discussions related to this work. S.C. was supported by a Hermann-Weyl-Instructorship from the Forschungsinstitut f\"ur Mathematik at ETH Z\"urich and the SNSF Ambizione grant PZ00P2\_223473. D.P.\ was supported by a Wallenberg Scholar fellowship. 

\subsection*{Disclaimer}

The results of this paper rely on Arthur's endoscopic classification \cite{arthurclassification}, which is not yet completely unconditional.
Thanks to \cite{AGIKMS} the only remaining result to be proved is the (standard and non-standard) weighted fundamental lemma for Lie algebras over positive characteristic local fields (generalizing \cite{chaudouard_laumon_wfl2}).
On a positive note, we understand that Connor Halleck-Dub\'e is making good progress on this generalization, so one can be hopeful that the results of \cite{arthurclassification} will soon become unconditional.

\section{Intersection cohomology of the Satake compactification via automorphic representations}\label{sec:Satake}
The \emph{Satake compactification} $A_g^{\Sat}$ is a highly singular compactification of $A_g$ as a stack over $\Spec \Z$. Set-theoretically, it admits a stratification
\begin{equation}\label{satakestratification}
    A_g^{\Sat}=\coprod_{0\leq k\leq g} A_k.
\end{equation}
On the level of the coarse moduli space over the complex numbers, it is the Proj of the ring of Siegel modular forms of degree $g$. In this section, we explain how to compute its low degree intersection cohomology with symplectic coefficients via automorphic representation theory. In particular, we prove Theorems \ref{thm:lowdegintcoh} and \ref{thm IH}.

Let $\lambda$ be a dominant weight of $\Sp_{2g}$, i.e.\ a tuple of integers $(\lambda_1,\dots,\lambda_g)$ such that $\lambda_1 \geq \lambda_2\geq \dots \geq \lambda_g \geq 0$.
Denote $|\lambda|=\lambda_1+\dots+\lambda_g$.
For each such $\lambda$, there is an irreducible representation $V_\lambda$ of $\mathrm{Sp}_{2g}$ (as algebraic group over $\Q$) of highest weight $\lambda$, which extends to a representation of $\GSp_{2g}$ by letting $z I_{2g} \in Z(\GSp_{2g})$ act by $z^{-|\lambda|}$.
We consider $\Qell \otimes_\Q V_\lambda$ as a continuous representation of $\GSp_{2g}(\Qell)$, inducing an $\ell$-adic local system on $A_g$, that we simply denote by $\V_{\lambda}$.
Then $\V_{\lambda}$ is a lisse sheaf of weight $|\lambda|$.
We will be interested in the intersection cohomology $\IH^\bullet(A_g^\Sat,\V_\lambda)$; that is, the cohomology of the intermediate extension of $\V_\lambda$ from $A_g$ to $A_g^\Sat$. 

We choose a field isomorphism $\iota: \C \simeq \Qellbar$ (only its restriction to the algebraic closure of $\Q$ actually plays a role) and fix it for the whole paper.
Let $\rho$ be the half-sum of positive roots. We set  $\tau = \lambda + \rho = (\lambda_1 + g, \dots, \lambda_g + 1)$. Theorem 4.7.2 in \cite{taibi_ecAn} (using Remarks 4.3.6 and 4.3.7 loc.\ cit.\ to reduce to the case where the central character of $V_\lambda$ is trivial, and purity) decomposes
\[ \bigoplus_k \Qellbar \otimes_{\Qell} \IH^k(A_g^\Sat, \V_\lambda) \]
as a (canonical) direct sum over a set $\tPsidut(\Sp_{2g})$ of formal Arthur--Langlands parameters
\[ \psi = \psi_0 \oplus \dots \oplus \psi_r = \pi_0[d_0] \oplus \dots \oplus \pi_r[d_r], \]
which are formal in the sense that $\psi$ is really defined as the set $\{(\pi_i,d_i) \,|\, 0 \leq i \leq r\}$ and the direct sum is only a suggestive notation.
Here, for each $i$, we have an integer $d_i \geq 1$, and a self-dual cuspidal level one algebraic representation $\pi_i \simeq \otimes'_v \pi_{i,v}$ for $\GL_{n_i,\Q}$. To belong to $\tPsidut(\Sp_{2g})$, a formal parameter $\psi$ has to satisfy the following additional condition: denoting by $(w_j(\pi_i))_{1 \leq j \leq n_i} \in \C^{n_i} / \mathfrak{S}_{n_i}$ the infinitesimal character of $\pi_{i,\infty}$, we have an equality of multi-sets
\[ \{ w_j(\pi_i) + (d_i-1)/2 - k \,|\, 0 \leq i \leq r,\, 1 \leq j \leq n_i,\, 0 \leq k \leq d_i-1 \} = \{ \pm(\lambda_1 + g), \dots, \pm (\lambda_g + 1), 0 \}. \]
Note that the right-hand side is really a set, so saying that this is an equality of multi-sets means that all entries on the left-hand side are distinct. In particular, for each $i$ the complex numbers $(w_j(\pi_i) + (d_i-1)/2)_{1 \leq j \leq n_i}$ are distinct integers, and it is natural to order the $w_j$ so that they satisfy $w_1(\pi_i) > \dots > w_{n_i}(\pi_i)$.
This condition on infinitesimal characters also implies that exactly one of the integers $n_i d_i$ is odd, and we may and do impose that this occurs for $i=0$.
Proposition 3.4.7 loc.\ cit.\ associates to each $\psi_i = \pi_i[d_i]$ a family of semisimple conjugacy classes $(\cpsc(\psi_i))_p$ (indexed by all primes $p$) in a certain group isomorphic to $\Spin_{n_i d_i}(\C)$, such that the image by the standard representation of $\cpsc(\psi_i)$ has eigenvalues
\[ \{ \alpha_j(\pi_{i,p}) p^{(d_i-1)/2-k} \,|\, 1 \leq j \leq n_i,\, 0 \leq k \leq d_i-1 \}\]
where $(\alpha_j(\pi_{i,p}))_{1 \leq j \leq n_i}$ are the eigenvalues of the Satake parameter of the unramified representation $\pi_{i,p}$ of $\GL_{n_i}(\Qp)$.
We know the Ramanujan conjecture in this case, i.e.\ that these eigenvalues $\alpha_j(\pi_{i,p})$ have absolute value $1$ (see \cite{Clozel_purity} or \cite{Caraiani_purity}).
The contribution $\sigma_{\psi,\iota}^\IH$ of $\psi$ to intersection cohomology is refined in \cite[Theorem 7.1.3]{taibi_ecAn}, but for the next lemma we only need to know that for any prime $p \neq \ell$ this contribution is unramified at $p$ and the characteristic polynomial of the geometric Frobenius $\Frob_p$ on it is a factor of that of
\[ \iota \left( p^{g(g+1)/4 + |\lambda|/2} \spin(\cpsc(\psi_0) \otimes \dots \otimes \spin(\cpsc(\psi_r)) \right). \]
Note that the eigenvalues of $\sigma_{\psi,\iota}^\IH(\Frob_p)$ are Weil numbers, either by purity of intersection cohomology, or by the fact that we know the Ramanujan conjecture in this case.

\begin{lem}
    The contribution of a parameter $\psi = \oplus_i \pi_i[d_i]$ to $\IH^\bullet(A_g^\Sat, \V_\lambda)$ vanishes in degree less than
    \[ g(g+1)/2 - \sum_{i=0}^r \frac{n_i \lfloor d_i^2 / 2 \rfloor}{4}. \]
\end{lem}
\begin{proof}
  By purity, it is enough to choose any prime $p \neq \ell$ and keep track of the weights of the eigenvalues of $\Frob_p$ in $\sigma_{\psi,\iota}^\IH$.
  More precisely, it is enough to show that the eigenvalues of each $\spin(\cpsc(\psi_i))$ have weight at least $-\tfrac 1 4 {n_i \lfloor d_i^2 / 2 \rfloor}$.
  We consider the three possible cases separately.
  We compute using the same parametrization of maximal tori in $\Spin$ groups as in \cite[\S 2.2]{taibi_ecAn}.
  To lighten the notation, we denote $\alpha_j = \alpha_j(\pi_{i,p})$.
  Up to renumbering, we may assume $\alpha_j \alpha_{n_i+1-j} = 1$ for all $1 \leq j \leq n_i$.
\begin{enumerate}
\item First assume that \(i=0\), i.e.\ \(n_i d_i\) is odd.
  Then $\cpsc(\psi_0)$ is conjugated to the  element of the maximal torus \(\Tcal_{\Spin_{n_i d_i}}\) of $\Spin_{n_i d_i}(\C)$ given by
  \[ (\alpha_1 p^{(d_i-1)/2}, \dots, \alpha_1 p^{(1-d_i)/2}, \dots, \alpha_{(n_i-1)/2} p^{(d_i-1)/2}, \dots, \alpha_{(n_i-1)/2} p^{(1-d_i)/2}, p^{(d_i-1)/2}, \dots, p, s), \]
  where
  \[ s^2 = \left(\prod_{j=1}^{(n_i-1)/2} \alpha_j \right)^{d_i} p^{(d_i^2-1)/8}. \]
  The eigenvalues of this element in the spin representations are
  \[ s \prod_{(j,k) \in S_1} \alpha_j p^k \prod_{k \in S_2} p^k \]
  as $S_1$ ranges over all subsets of
  \[ \left\{ 1, \dots, \frac{n_i-1}{2} \right\} \times \left\{ \frac{1-d_i}{2}, \dots, \frac{d_i-1}{2} \right\} \]
  and $S_2$ ranges over all subsets of $\{-1, \dots, (1-d_i)/2\}$.
  We see that the smallest weight occurring is
  \[ \frac{d_i^2-1}{8} - \frac{n_i-1}{2} \sum_{k=1}^{(d_i-1)/2} 2k - \sum_{k=1}^{(d_i-1)/2} 2k = \frac{-n_i (d_i^2-1)}{8}. \]
\item Now assume that \(i>0\) (i.e.\ \(n_i\) even) and \(d_i\) is odd.
  Then $\cpsc(\psi_0)$ is conjugated (by the orthogonal group) to the following element of \(\Tcal_{\Spin_{n_i d_i}}\)
  \[ (\alpha_1 p^{(d_i-1)/2}, \dots, \alpha_1 p^{(1-d_i)/2}, \dots, \alpha_{n_i/2} p^{(d_i-1)/2}, \dots, \alpha_{n_i/2} p^{(1-d_i)/2}, s) \]
  where $s^2 = (\prod_j \alpha_j)^{d_i}$.
  Its eigenvalues in the spin representation are
  \[ s \prod_{(j,k) \in S} \alpha_j p^k \]
  where $S$ ranges over all subsets of
  \[ \left\{ 1, \dots, \frac{n_i}{2} \right\} \times \left\{ \frac{1-d_i}{2}, \dots, \frac{d_i-1}{2} \right\}. \]
  We see that the smallest weight occurring is
  \[ -2 \cdot \frac{n_i}{2} \sum_{k=1}^{(d_i-1)/2} k = \frac{-n_i (d_i^2-1)}{8}. \]
\item The case where \(d_i\) is even (thus $i>0$ and $n_i$ is even) is similar and we omit details: we find that the smallest weight occurring is
  \[ \frac{-n_i d_i^2}{8}. \qedhere \]
\end{enumerate}
\end{proof}

We are thus interested in determining, for some fixed $M$, the set of all parameters $\psi \in \tPsidut(\Sp_{2g})$ (for varying $g$ and $\lambda$) satisfying
\begin{equation} \label{eq:first_ineq_M}
  g(g+1)/2 - \sum_i \frac{n_i \lfloor d_i^2 / 2 \rfloor}{4} + |\lambda| \leq M.
\end{equation}
(This will eventually allow us to understand intersection cohomology $\IH^p(A_g^\Sat,\V_\lambda)$ for $p+|\lambda| \leq M$.)
Recall that the vector $\tau$ consists precisely of the positive elements of the set
\[ \{ w_j(\pi_i) + (d_i-1)/2 - k \,|\, 0 \leq i \leq r,\, 1 \leq j \leq n_i,\, 0 \leq k \leq d_i-1 \} . \]
It is convenient to denote by $\tau^{(i)}$ the subtuple containing only the positive elements of the subset
\[ \{ w_j(\pi_i) + (d_i-1)/2 - k \,|\, 1 \leq j \leq n_i,\, 0 \leq k \leq d_i-1 \}. \]
Then $g(g+1)/2 + |\lambda| = |\tau| = \sum_{i=0}^r |\tau^{(i)}|$, so condition \eqref{eq:first_ineq_M} becomes
\begin{equation} \label{eq:refined_ineq_M}
  \sum_{i=0}^r k(\psi_i) \leq M \ \ \ \text{where} \ \ k(\psi_i) := |\tau^{(i)}| - \frac{n_i \lfloor d_i^2 / 2 \rfloor}{4}.
\end{equation}

\begin{lem} \label{lem:k_psi_i_ineq}
  Take $\psi = \psi_0 \oplus \dots \oplus \psi_r \in \tPsidut(\Sp_{2g})$. For each $0 \leq i \leq r$, denoting $r_i = \lfloor n_i/2 \rfloor$, we have
  \[ k(\psi_i) \geq \sum_{j=1}^{r_i} w_j(\pi_i) \geq r_i(r_i+1)/2 \ \ \ \text{if } d_i \text{ is odd}, \]
  \[ k(\psi_i) \geq 2 \sum_{j=1}^{r_i} w_j(\pi_i) - r_i \geq 2 r_i^2 \ \ \ \text{if } d_i \text{ is even}. \]
  In particular, we always have $k(\psi_i) \geq 0$.
\end{lem}
\begin{proof}
  In the proof we set $w_j = w_j(\pi_i)$ to lighten the notation.
  \begin{enumerate}
    \item We first consider the case where $n_i$ and $d_i$ are even, so that $r_i = n_i/2$.
      We have
      \[ |\tau^{(i)}| - \frac{n_i d_i^2}{8} = \frac{d_i}{2} \left(2 \sum_{j=1}^{r_i} w_j - \frac{r_i d_i}{2} \right). \]
      Let $R$ be the set of even integers $d \geq 2$ such that $\pi_i[d]$ is regular, i.e.\ such that the multi-set
      \[ \{ w_j + (d-1)/2 - k \,|\, 1 \leq j \leq n_i,\, 0 \leq k \leq d-1 \} \]
      is a set.
      It is of the form $\{2, \dots, d_{\max}\}$ and contains $d_i$.
      For $d \in R$ and $0 \leq j \leq r_i-1$ we have $2w_{r_i-j} \geq (2j+1)d+1$, implying
      \begin{equation} \label{eq:ineq_sum_wj_d_in_R}
        2 \sum_{j=1}^{r_i} w_j \geq r_i (d+1) + r_i (r_i-1) d = d r_i^2 + r_i.
      \end{equation}
      We deduce
      \[ \frac{4}{r_i} \sum_{j=1}^{r_i} w_j - d_{\max} \geq d_{\max}(2r_i-1) + 2 > 2, \]
      which implies that the minimum of
      \[ d \longmapsto d \left(2 \sum_{j=1}^{r_i} w_j - \frac{r_i d}{2} \right) = \frac{r_i}{2} d \left( 4 \frac{\sum_{j=1}^{r_i} w_j}{r_i} - d \right) \]\
      on $R$ is reached for $d=2$.
      Thus
      \[ |\tau^{(i)}|-\frac{n_i d_i^2}{8} \geq 2 \sum_{j=1}^{r_i} w_j - r_i \geq \frac{r_i(r_i+1)}{2}, \]
      where the last inequality follows from \eqref{eq:ineq_sum_wj_d_in_R} for $d=2$.
    \item Next consider the case where $d_i$ is odd and $n_i = 2r_i$ is even (because of the level one assumption $r_i$ is even).
      We compute
      \[ |\tau^{(i)}| - n_i \frac{d_i^2-1}{8} = \frac{d_i}{2} \left( 2\sum_{j=1}^{r_i} w_j - \frac{r_id_i}{2} \right) + r_i/4. \]
      Let $R$ be the set of odd integers such that $\pi[d]$ is regular. It is of the form $\{1, \dots, d_{\max}\}$.
      Similarly to the first case we obtain for $d \in R$ and $0 \leq j \leq r_i-1$ the inequality $2w_{r_i-j} \geq (2j+1) d_i + 1$, implying for $d \in R$
      \[ 2\sum_{j=1}^{r_i} w_j \geq d r_i^2 + r_i. \]
      As in the first case, this implies that the minimum of
      \[ d \longmapsto \frac{d}{2} \left( 2 \sum_{j=1}^{r_i} w_j - \frac{r_i d}{2} \right) \]
      on $R$ is reached for $d=1$ and so
      \[ |\tau^{(i)}| - n_i \frac{d_i^2-1}{8} \geq \sum_{j=1}^{r_i} w_j \geq \frac{r_i(r_i+1)}{2}. \]
    \item Finally consider the case where $n_i = 2r_i+1$ and $d_i$ are odd.
      We again compute
      \[ |\tau^{(i)}| - n_i \frac{d_i^2-1}{8} = d_i \left( \sum_{j=1}^{r_i} w_j - \frac{r_id_i}{4} \right) + \frac{r_i}{4}. \]
      For $r_i=0$ (i.e.\ $\pi_i$ the trivial representation for $\mathrm{GL}_1$) this is zero, so we may assume $r_i>0$.
      Let $R = \{1,\dots,d_{\max}\}$ be the set of odd integers $d$ such that $\pi_i[d]$ is regular.
      For $d \in R$ and $0 \leq j \leq r_i-1$ we have $w_{r_i-j} \geq (j+1) d$ and so
      \[ \sum_{j=1}^{r_i} w_j \geq d_{\max} \frac{r_i(r_i+1)}{2}, \]
      which implies
      \[ \frac{4}{r_i} \sum_{j=1}^{r_i} w_j - d_{\max} \geq n_i d_{\max} > 1. \]
      As in the previous cases we deduce
      \[ |\tau^{(i)}| - n_i \frac{d_i^2-1}{8} \geq \sum_{j=1}^{r_i} w_j \geq \frac{r_i(r_i+1)}{2}. \qedhere\]
\end{enumerate}
\end{proof}

\begin{prop} \label{pro:params_up_to_23}
  The list of $\psi \in \tPsidut(\Sp_{2g})$ not of the form $[2g+1]$ satisfying \eqref{eq:first_ineq_M} for $M=23$ is given in the third column of Tables \ref{tab:rest}, \ref{tab:C}, and \ref{tab:DEF}.
\end{prop}
\begin{proof}
  We classify all $\psi \in \tPsidut(\Sp_{2g})$ satisfying condition \eqref{eq:first_ineq_M} for $M=23$ using Lemma \ref{lem:k_psi_i_ineq}.
  This lemma first implies $r_i \leq 3$ for $d_i$ even and $r_i \leq 6$ for $d_i$ odd.
  \begin{itemize}
  \item Consider first the case where $d_i$ is even.
    If $r_i=3$ then $\sum_{j=1}^{3} w_j(\pi_i) \leq 13$, implying $w_1(\pi_i) \leq \tfrac{21}{2}$.
    It follows from \cite[Theorem F]{chenevierlannes} that the case $r_i=3$ does not occur.
    If $r_i=2$ then $w_1(\pi_i) + w_2(\pi_i) \leq 12$, and perusing the table of self-dual algebraic regular cuspidal automorphic representations for $\GL_{4,\Q}$ of symplectic type up to weight $23$ shows
    \[ \pi_i \in \{ \Delta_{19,7}, \Delta_{21,5}, \Delta_{21,9}, \Delta_{21,13} \}, \]
    but none of these representations satisfies $\sum_{j=1}^2 w_j \leq 12$, so the case $r_i=2$ is ruled out as well.
    Finally, for $r_i=1$ we find that $\pi_i$ corresponds to an eigenform in $S_{2k}(\SL_2(\Z))$ for $2k \leq 24$, i.e.\ $\pi_i$ is either $\Delta_w$ for $w \in \{11,15,17,19,21\}$ or one of the two $\Delta_{23}^{(j)}$, $j \in \{1,2\}$.
    We have $d_i=2$ except for $\Delta_{11}$, where we can also have $d_i=4$.
  \item Next we consider the case where $d_i$ is odd, excluding the case $\pi_i = 1$ (i.e.\ $r_i=0$).
    We now have $\sum_{j=1}^{r_i} w_j(\pi_i) \leq M$, in particular $w_1(\pi_i) \leq M - r_i(r_i-1)/2$.
    For $r_i \geq 6$ this imposes $w_1(\pi_i) \leq 8$ and this does not occur by \cite[Theorem F]{chenevierlannes}.
    For $r_i \leq 5$ perusing the tables of self-dual algebraic regular cuspidal automorphic representations for $\GL_{n_i,\Q}$ of orthogonal type (\cite{taibi} or \cite{cheneviertaibi}) shows that there is no such representation satisfying $\sum_{j=1}^{r_i} w_j(\pi_i) \leq 25$ if $r_i>3$.
    (Presumably the cases $r_i \in \{4,5\}$ could also be ruled out using the method of \cite[\S 2.4]{cheneviertaibi}.)
    For $r_i \leq 3$ we find the following possibilities for $\pi_i$ (see \cite[\S 4]{chenevierrenard} and \cite[\S 3.5]{taibi_ecAn} for the three automorphic lifts involved):
    \begin{itemize}
    \item $\pi_i = \Sym^2 \Delta_{w}$ for $n_i = 3$, where $w \in \{11,15,17,19,21,23\}$ (and $\Delta_{23}$ is understood as either one of $\Delta_{23}^{(j)}$, $j \in \{1,2\}$),
    \item $\pi_i = \Delta_{w_1} \otimes \Delta_{w_2}$ for $n_i=4$, where $23 \geq w_1 > w_2$ (idem),
    \item $\pi_i = \Lambda^*(\Delta_{w_1,w_2})$ where $23 \geq w_1$.
    \end{itemize}
    We find that the only possibility is $d_i=1$ in all these cases.
  \end{itemize}
  To sum up, the constituents of $\psi$ have to be in the following list
  \begin{itemize}
  \item $[2d+1]$ for $d \geq 0$,
  \item $\Delta_w[2]$ for $w \in \{11,15,17,19,21,23\}$,
  \item $\Delta_{11}[4]$,
  \item $(\Sym^2 \Delta_w)[1]$ for $w \in \{11,15,17,19,21,23\}$,
  \item $(\Delta_{w_1} \otimes \Delta_{w_2})[1]$ for $23 \geq w_1 > w_2$,
  \item $\Lambda^*(\Delta_{w_1,w_2})[1]$ for
    \[ (w_1,w_2) \in \{ (19,7), (21,5), (21,9), (21,13), (23,7), (23,9), (23,13). \} \]
  \end{itemize}
  It only remains to combine some of these constituents (with distinct weights and exactly one odd-dimensional constituent) to form parameters $\psi \in \tPsidut(\Sp_{2g})$ and only keep those satisfying \eqref{eq:refined_ineq_M} for $M=23$.
\end{proof}

Now we want to compute the contribution to intersection cohomology in each case.
For $d \geq 0$ denote
\[ T_\ell \langle d \rangle := \bigotimes_{i=1}^d (\Qell \oplus \Qell(-i)). \]
Note that this is the contribution of the ``trivial" parameter $[2d+1]$ to $\IH^\bullet(A_d^\Sat, \Qell)$.
Denote by $S_\ell\langle k\rangle$ the $\ell$-adic Galois representation attached to cusp forms for $\SL_2(\Z)$ of weight $k$, and by $s\langle k \rangle = \tfrac{1}{2} \dim_{\Qell} S_\ell \langle k \rangle$ the dimension of this space of cusp forms.
Similarly for odd integers $w_1 > w_2 > 0$ denote by $S_{\ell} \langle \tfrac{w_1+w_2-4}{2}, \tfrac{w_1-w_2-2}{2} \rangle$ the direct sum of the Galois representations associated to the self-dual cuspidal automorphic representations $\Delta_{w_1,w_2}^{(i)}$ for $\PGL_{4,\Q}$ with infinitesimal character $(\pm \tfrac{w_1}{2}, \pm \tfrac{w_2}{2})$ at the real place.
In the notation of \cite[Theorem 6.3.1]{taibi_ecAn}
\[ \Qellbar \otimes_{\Qell} S_\ell \langle \tfrac{w_1+w_2-4}{2}, \tfrac{w_1-w_2-2}{2} \rangle \simeq \bigoplus_i (\std \circ \rho_{\Delta_{w_1,w_2}^{(i)},\iota}^{\GSp})(\tfrac{1-w_1}{2}). \]
(The right-hand side a priori only defines a representation over $\Qellbar$ but it may be cut out using rational Hecke operators inside $\Qellbar \otimes_{\Qell} \IH^3(A_2^\Sat, \V_\lambda)$ for $\lambda = (\tfrac{w_1+w_2-4}{2}, \tfrac{w_1-w_2-2}{2})$, and is thus defined over $\Qell$.)
Note that the normalization (Tate twist) is so that this Galois representation is effective (its Hodge-Tate weights are non-negative), and it is pure of weight $w_1$.

We also need to recall the definition of the signs $(u_i(\psi))_{1 \leq i \leq r}$ associated to $\psi = \psi_0 \oplus \dots \oplus \psi_r \in \tPsidut(\Sp_{2g})$ appearing in \cite[Theorems 4.7.2 and 7.1.3]{taibi_ecAn}.
For $1 \leq i \leq r$ recall that $n_i$ is even and let $J_i$ be the subset of $\{1, \dots, g\}$ such that we have
\[ \{ w_j(\pi_i) + \tfrac{d_i-1}{2} - k \,|\, 1 \leq j \leq \tfrac{n_i}{2}, \, 0 \leq k \leq d_i-1 \} = \{ \tau_j \,|\, j \in J_i \}. \]
Denoting by $f_i$ the cardinality of $\{j \in J_i \,|\, j \text{ even}\}$ we have
\[ u_i(\psi) = (-1)^{f_i} \prod_{\substack{0 \leq j \leq r \\ d_i+d_j \text{ odd}}} \epsilon(\tfrac{1}{2}, \pi_i \times \pi_j)^{\min(d_i,d_j)}. \]
Recall from \cite[\S 3.9]{chenevierrenard} that the symplectic root number $\epsilon(\tfrac{1}{2}, \pi_i \times \pi_j)$ is easily computed from the weights of $\pi_i$ and $\pi_j$.
We also require the half-spin representations $\spin_{\psi_i}^{\pm}$ distinguished in \cite[Definition 4.7.1]{taibi_ecAn}.

\begin{thm}[Theorem \ref{thm IH}] \label{thm IH tables}
  The contributions of all $\psi \in \tPsidut(\Sp_{2g}) \smallsetminus \{[2g+1]\}$ to $\IH^k(A_g^\Sat, \V_\lambda)$ for $k + |\lambda| \leq 23$ are given in the last column of Tables \ref{tab:rest}, \ref{tab:C}, and \ref{tab:DEF}, where we have written $L^i$ for $\Qell(-i)$.
\end{thm}
\begin{proof}
We compute more precisely the contribution in each case.
\begin{enumerate}[(Type A)]
\item (Elliptic holomorphic cusp forms) The contribution of parameters of type $(\Sym^2 \Delta_w)[1]$ (here $g=1$ and $\lambda = w-1$) to $\IH^\bullet(A_g^\Sat, \V_\lambda)$ is $S_\ell\langle w+1\rangle$.
\item (Genus two holomorphic Siegel cusp forms of stable type) The contribution of parameters of type $(\Lambda^* \Delta_{w_1,w_2})[1]$ (here $g=2$ and $\lambda = (\tfrac{w_1+w_2-4}{2}, \tfrac{w_1-w_2-2}{2})$) to $\IH^\bullet(A_g^\Sat, \V_\lambda)$ is $S_\ell\langle\lambda\rangle$ (see \cite[Remark 5.2.3]{taibi_ecAn}).
\item \label{it:typeC} Suppose $\psi = \Delta_{w}[2] \oplus [2d+1]$ for $0 \leq d \leq (w-3)/2$.
  We have
  \[ u_1(\psi) = -\epsilon(\tfrac{1}{2}, \Delta_{w})^{\min(2,2d+1)} =
    \begin{cases}
      -1 & \text{ if } d>0 \text{ or } w \equiv -1 \mod 4 \\
      +1 & \text{ if } d=0 \text{ and } w \equiv 1 \mod 4.
    \end{cases}
  \]
  We find (similar computation as in \cite[\S 9.2]{taibi_ecAn} for the case $w=11$, using Corollary 6.2.3 loc.\ cit.) that the contribution of such parameters to $\IH^\bullet(A_g^\Sat, \V_\lambda)$ is
  \[ (\Qell(\tfrac{-w+1} 2) \oplus \Qell(\tfrac{-w-1} 2)) \otimes T_\ell \langle d \rangle \]
  in the first case, and $S_\ell\langle w+1 \rangle$ in the second case.
\item Consider $\psi = \Delta_{11}[2] \oplus (\Sym^2 \Delta_{11})[1]$, so $g=3$ and $\lambda = (8,4,4)$.
  Then $u_1(\psi) = -\epsilon(\tfrac{1}{2}, \Delta_{11} \times \Sym^2 \Delta_{11}) = +1$ and a computation similar to the ones in Type C and E shows that the contribution of $\psi$ to $\IH^\bullet(A_g^\Sat, \V_\lambda)$ is
  \[ S_\ell \langle 12\rangle^{\otimes 2} \simeq \Sym^2 S_\ell\langle12\rangle \oplus \Qell(-11). \]
\item \label{it:typeE}
  The contribution of parameters of type $\Delta_{11}[4] \oplus [2d+1]$ for $d \in \{0,1,2,3\}$ (here $g=d+4$ and $\lambda=(7-g,7-g,7-g,7-g,0,\dots,0)$) to $\IH^\bullet(A_g^\Sat, \V_\lambda)$ is
  \[ \left( \Sym^2 S_\ell\langle12\rangle \oplus \bigoplus_{i=9}^{13} \Qell(-i) \right) \otimes T_\ell \langle d \rangle \]
  as explained in \cite[\S 9.2]{taibi_ecAn}.
\item Consider $\psi = (\Delta_{w_1} \otimes \Delta_{w_2})[1] \oplus [2d+1]$ where $w_1>w_2$ and $0 \leq d \leq \tfrac{w_1-w_2}{2}-1$.
  Here $g=d+2$ and
  \[ \lambda = \left(\frac{w_1+w_2}{2}-d-2, \frac{w_1-w_2}{2}-d-1, 0, \dots, 0 \right). \]
  We have $u_1(\psi) = -\epsilon(\tfrac{1}{2}, \Delta_{w_1} \times \Delta_{w_2}) = -1$ and so the contribution of these parameters to $\IH^\bullet(A_g^\Sat, \V_\lambda)$ is
  \[ S_\ell \langle w_2+1 \rangle^{\oplus s\langle w_1+1 \rangle }((w_2-w_1)/2) \otimes T_\ell \langle d \rangle. \qedhere \]
\end{enumerate}
\end{proof}

\begin{proof}[Proof of Theorem \ref{thm:lowdegintcoh}]
  We first prove the theorem up to semi-simplification.
  For $g \neq 7$ this follows from Theorem \ref{thm IH tables} and selecting the rows with $\lambda=0$ in the tables.
  For $g=7$ the semi-simplification of $\IH^\bullet(A_7^\Sat, \Qell)$ was already computed in \cite[\S 9.2]{taibi_ecAn}: only two parameters contribute, namely $[15]$ and $\Delta_{11}[4] \oplus [7]$ \ref{it:typeE}.
  We deduce
  \[ \bigoplus_k \IH^k(A_7^\Sat, \Qell)^\mathrm{ss} \simeq T_\ell \langle 7 \rangle \oplus \left( \Sym^2 S_\ell \langle 12 \rangle \oplus \bigoplus_{i=9}^{13} \Qell(-i) \right) \otimes T_\ell \langle 3 \rangle \]
  and in weight $\leq 23$ the only non-Tate part on the right-hand side is in weight $22$.

  To remove the semi-simplifications we use that $\IH^k(A_g^\Sat, \Qell)$ is unramified away from $\ell$ and crystalline at $\ell$.
  In the cases of Tate type this implies that $\IH^k(A_g^\Sat, \Qell)(\tfrac{k}{2})$ is an everywhere unramified Galois representation, so it is trivial.
  For the case $(g,k)=(7,22)$ we know thanks to \cite[Theorem B (1)]{newtonthorne_adjselmer} that there is no non-trivial conductor one geometric extension between $\Sym^2 S_\ell \langle 12 \rangle$ and $\Qell(-11)$.
\end{proof}

\begin{table}[]
    \centering
    \begin{tabular}{c|c|c|c}
      $g$ & $\lambda$ & $\psi$ & contribution in weight $\leq 23$ \\
    $1$ & $(10)$ &  $(\Sym^2 \Delta_{11})[1]$ & $S_\ell \langle 12 \rangle$ \\
    $1$ & $(14)$ &  $(\Sym^2 \Delta_{15})[1]$ & $S_\ell \langle 16 \rangle$ \\
    $1$ & $(16)$ &  $(\Sym^2 \Delta_{17})[1]$ & $S_\ell \langle 18 \rangle$ \\
    $1$ & $(18)$ &  $(\Sym^2 \Delta_{19})[1]$ & $S_\ell \langle 20 \rangle$ \\
    $1$ & $(20)$ &  $(\Sym^2 \Delta_{21})[1]$ & $S_\ell \langle 22 \rangle$\\
    $1$ & $(22)$ &  $(\Sym^2 \Delta_{23})[1]$ & $S_\ell \langle 24 \rangle$\\
    $2$ & $(11,5)$ &  $\Lambda^* (\Delta_{19,7})[1]$ & $S_\ell \langle 11, 5 \rangle$ \\
    $2$ & $(11,7)$ &  $\Lambda^* (\Delta_{21,5})[1]$ & $S_\ell \langle 11, 7 \rangle$ \\
    $2$ & $(13,5)$ &  $\Lambda^* (\Delta_{21,9})[1]$ & $S_\ell \langle 13, 5 \rangle$ \\
    $2$ & $(13,7)$ &  $\Lambda^* (\Delta_{23,7})[1]$ & $S_\ell \langle 13, 7 \rangle$ \\
    $2$ & $(14,6)$ &  $\Lambda^* (\Delta_{23,9})[1]$ & $S_\ell \langle 14, 6 \rangle$ \\
    $2$ & $(15,3)$ &  $\Lambda^* (\Delta_{21,13})[1]$ & $S_\ell \langle 15, 3 \rangle$ \\
    $2$ & $(16,4)$ &  $\Lambda^* (\Delta_{23,13})[1]$ & $S_\ell \langle 16, 4 \rangle$ \\
    
    \end{tabular}\vspace{.1in}

    \label{tab:rest}
    \caption{Types A and B}
    
\end{table}

\begin{table}[]
    \centering
    \begin{tabular}{c|c|c|c}
      $g$ & $\lambda$ & $\psi$ & contribution in weight $\leq 23$ \\
  $2$ & $(4, 4)$ & $\Delta_{11}[2] \oplus [1]$ & $L^{5} \oplus L^{6}$ \\
  $3$ & $(3, 3, 0)$ & $\Delta_{11}[2] \oplus [3]$ & $L^{5} \oplus 2L^{6} \oplus L^{7}$ \\
  $4$ & $(2, 2, 0, 0)$ & $\Delta_{11}[2] \oplus [5]$ & $L^{5} \oplus 2L^{6} \oplus 2L^{7} \oplus 2L^{8} \oplus L^{9}$ \\
  $5$ & $(1, 1, 0, 0, 0)$ & $\Delta_{11}[2] \oplus [7]$ & $L^{5}\oplus2L^{6}\oplus2L^{7}\oplus3L^{8} \oplus 3L^{9} \oplus 2L^{10} \oplus 2L^{11}$ \\
  $6$ & $(0, 0, 0, 0, 0, 0)$ & $\Delta_{11}[2] \oplus [9]$ & $L^{5} \oplus 2L^{6} \oplus 2L^{7} \oplus 3L^{8} \oplus 4L^{9} \oplus 4L^{10} \oplus 4L^{11}$ \\
  $2$ & $(6, 6)$ & $\Delta_{15}[2] \oplus [1]$ & $L^{7} \oplus L^{8}$ \\
  $3$ & $(5, 5, 0)$ & $\Delta_{15}[2] \oplus [3]$ & $L^{7} \oplus 2L^{8} \oplus L^{9}$ \\
  $4$ & $(4, 4, 0, 0)$ & $\Delta_{15}[2] \oplus [5]$ & $L^{7} \oplus 2L^{8} \oplus 2L^{9} \oplus 2L^{10} \oplus L^{11}$ \\
  $5$ & $(3, 3, 0, 0, 0)$ & $\Delta_{15}[2] \oplus [7]$ & $L^{7} \oplus 2L^{8} \oplus 2L^{9} \oplus 3L^{10} \oplus 3L^{11}$ \\
  $6$ & $(2, 2, 0, 0, 0, 0)$ & $\Delta_{15}[2] \oplus [9]$ & $L^{7} \oplus 2L^{8} \oplus 2L^{9} \oplus 3L^{10} \oplus 4L^{11}$ \\
  $7$ & $(1, 1, 0, 0, 0, 0, 0)$ & $\Delta_{15}[2] \oplus [11]$ & $L^{7} \oplus 2L^{8} \oplus 2L^{9} \oplus 3L^{10} \oplus 4L^{11}$ \\
  $8$ & $(0, 0, 0, 0, 0, 0, 0, 0)$ & $\Delta_{15}[2] \oplus [13]$ & $L^{7} \oplus 2L^{8} \oplus 2L^{9} \oplus 3L^{10} \oplus 4L^{11}$ \\
  $2$ & $(7, 7)$ & $\Delta_{17}[2] \oplus [1]$ & $S_\ell\langle18\rangle$ \\
  $3$ & $(6, 6, 0)$ & $\Delta_{17}[2] \oplus [3]$ & $L^{8} \oplus 2L^{9} \oplus L^{10}$ \\
  $4$ & $(5, 5, 0, 0)$ & $\Delta_{17}[2] \oplus [5]$ & $L^{8} \oplus 2L^{9} \oplus 2L^{10} \oplus 2L^{11}$ \\
  $5$ & $(4, 4, 0, 0, 0)$ & $\Delta_{17}[2] \oplus [7]$ & $L^{8} \oplus 2L^{9} \oplus 2L^{10} \oplus 3L^{11}$ \\
  $6$ & $(3, 3, 0, 0, 0, 0)$ & $\Delta_{17}[2] \oplus [9]$ & $L^{8} \oplus 2L^{9} \oplus 2L^{10} \oplus 3L^{11}$ \\
  $7$ & $(2, 2, 0, 0, 0, 0, 0)$ & $\Delta_{17}[2] \oplus [11]$ & $L^{8} \oplus 2L^{9} \oplus 2L^{10} \oplus 3L^{11}$ \\
  $8$ & $(1, 1, 0, 0, 0, 0, 0, 0)$ & $\Delta_{17}[2] \oplus [13]$ & $L^{8} \oplus 2L^{9} \oplus 2L^{10} \oplus 3L^{11}$ \\
  $9$ & $(0, 0, 0, 0, 0, 0, 0, 0, 0)$ & $\Delta_{17}[2] \oplus [15]$ & $L^{8} \oplus 2L^{9} \oplus 2L^{10} \oplus 3L^{11}$ \\
  $2$ & $(8, 8)$ & $\Delta_{19}[2] \oplus [1]$ & $L^{9} \oplus L^{10}$ \\
  $3$ & $(7, 7, 0)$ & $\Delta_{19}[2] \oplus [3]$ & $L^{9} \oplus 2L^{10} \oplus L^{11}$ \\
  $4$ & $(6, 6, 0, 0)$ & $\Delta_{19}[2] \oplus [5]$ & $L^{9} \oplus 2L^{10} \oplus 2L^{11}$ \\
  $5$ & $(5, 5, 0, 0, 0)$ & $\Delta_{19}[2] \oplus [7]$ & $L^{9} \oplus 2L^{10} \oplus 2L^{11}$ \\
  $6$ & $(4, 4, 0, 0, 0, 0)$ & $\Delta_{19}[2] \oplus [9]$ & $L^{9} \oplus 2L^{10} \oplus 2L^{11}$ \\
  $7$ & $(3, 3, 0, 0, 0, 0, 0)$ & $\Delta_{19}[2] \oplus [11]$ & $L^{9} \oplus 2L^{10} \oplus 2L^{11}$ \\
  $8$ & $(2, 2, 0, 0, 0, 0, 0, 0)$ & $\Delta_{19}[2] \oplus [13]$ & $L^{9} \oplus 2L^{10} \oplus 2L^{11}$ \\
  $9$ & $(1, 1, 0, 0, 0, 0, 0, 0, 0)$ & $\Delta_{19}[2] \oplus [15]$ & $L^{9} \oplus 2L^{10} \oplus 2L^{11}$ \\
  $10$ & $(0, 0, 0, 0, 0, 0, 0, 0, 0, 0)$ & $\Delta_{19}[2] \oplus [17]$ & $L^{9} \oplus 2L^{10} \oplus 2L^{11}$ \\
  $2$ & $(9, 9)$ & $\Delta_{21}[2] \oplus [1]$ & $S_\ell\langle22\rangle$ \\
  $3$ & $(8, 8, 0)$ & $\Delta_{21}[2] \oplus [3]$ & $L^{10} \oplus 2L^{11}$ \\
  $4$ & $(7, 7, 0, 0)$ & $\Delta_{21}[2] \oplus [5]$ & $L^{10} \oplus 2L^{11}$ \\
  $5$ & $(6, 6, 0, 0, 0)$ & $\Delta_{21}[2] \oplus [7]$ & $L^{10} \oplus 2L^{11}$ \\
  $6$ & $(5, 5, 0, 0, 0, 0)$ & $\Delta_{21}[2] \oplus [9]$ & $L^{10} \oplus 2L^{11}$ \\
  $7$ & $(4, 4, 0, 0, 0, 0, 0)$ & $\Delta_{21}[2] \oplus [11]$ & $L^{10} \oplus 2L^{11}$ \\
  $8$ & $(3, 3, 0, 0, 0, 0, 0, 0)$ & $\Delta_{21}[2] \oplus [13]$ & $L^{10} \oplus 2L^{11}$ \\
  $9$ & $(2, 2, 0, 0, 0, 0, 0, 0, 0)$ & $\Delta_{21}[2] \oplus [15]$ & $L^{10} \oplus 2L^{11}$ \\
  $10$ & $(1, 1, 0, 0, 0, 0, 0, 0, 0, 0)$ & $\Delta_{21}[2] \oplus [17]$ & $L^{10} \oplus 2L^{11}$ \\
  $11$ & $(0, 0, 0, 0, 0, 0, 0, 0, 0, 0, 0)$ & $\Delta_{21}[2] \oplus [19]$ & $L^{10} \oplus 2L^{11}$ \\
  $2$ & $(10, 10)$ & $\Delta_{23}[2] \oplus [1]$ & $L^{11}$ \\
  $3$ & $(9, 9, 0)$ & $\Delta_{23}[2] \oplus [3]$ & $L^{11}$ \\
  $4$ & $(8, 8, 0, 0)$ & $\Delta_{23}[2] \oplus [5]$ & $L^{11}$ \\
  $5$ & $(7, 7, 0, 0, 0)$ & $\Delta_{23}[2] \oplus [7]$ & $L^{11}$ \\
  $6$ & $(6, 6, 0, 0, 0, 0)$ & $\Delta_{23}[2] \oplus [9]$ & $L^{11}$ \\
  $7$ & $(5, 5, 0, 0, 0, 0, 0)$ & $\Delta_{23}[2] \oplus [11]$ & $L^{11}$ \\
  $8$ & $(4, 4, 0, 0, 0, 0, 0, 0)$ & $\Delta_{23}[2] \oplus [13]$ & $L^{11}$ \\
  $9$ & $(3, 3, 0, 0, 0, 0, 0, 0, 0)$ & $\Delta_{23}[2] \oplus [15]$ & $L^{11}$ \\
  $10$ & $(2, 2, 0, 0, 0, 0, 0, 0, 0, 0)$ & $\Delta_{23}[2] \oplus [17]$ & $L^{11}$ \\
  $11$ & $(1, 1, 0, 0, 0, 0, 0, 0, 0, 0, 0)$ & $\Delta_{23}[2] \oplus [19]$ & $L^{11}$ \\
  $12$ & $(0, 0, 0, 0, 0, 0, 0, 0, 0, 0, 0, 0)$ & $\Delta_{23}[2] \oplus [21]$ & $L^{11}$
    \end{tabular}\vspace{.1in}

    \caption{Type C}
    \label{tab:C}
\end{table}

\begin{table}[]
    \centering
    \begin{tabular}{c|c|c|c}
      $g$ & $\lambda$ & $\psi$ & contribution in weight $\leq 23$ \\
  $3$ & $(8,4,4)$ &  $\Delta_{11}[2] \oplus (\Sym^2 \Delta_{11})[1]$ & $\Sym^2 S_\ell \langle 12 \rangle \oplus L^{11}$ \\
 $4$ & $(3, 3, 3, 3)$ & $\Delta_{11}[4] \oplus [1]$ & $\Sym^2 S_\ell \langle 12 \rangle \oplus L^{9} \oplus L^{10} \oplus L^{11}$ \\
  $5$ & $(2, 2, 2, 2, 0)$ & $\Delta_{11}[4] \oplus [3]$ & $\Sym^2 S_\ell \langle 12 \rangle \oplus L^{9} \oplus 2L^{10} \oplus 2L^{11}$ \\
  $6$ & $(1, 1, 1, 1, 0, 0)$ & $\Delta_{11}[4] \oplus [5]$ & $\Sym^2 S_\ell \langle 12 \rangle \oplus L^{9} \oplus 2L^{10} \oplus 3L^{11}$ \\
  $7$ & $(0, 0, 0, 0, 0, 0, 0)$ & $\Delta_{11}[4] \oplus [7]$ & $\Sym^2 S_\ell \langle 12 \rangle \oplus L^{9} \oplus 2L^{10} \oplus 3L^{11}$ \\

  $2$ & $(11, 1)$ & $\Delta_{15} \otimes \Delta_{11} \oplus [1]$ & $S_\ell \langle 12 \rangle \otimes L^{2}$ \\
  $3$ & $(10, 0, 0)$ & $\Delta_{15} \otimes \Delta_{11} \oplus [3]$ & $S_\ell \langle 12 \rangle \otimes ( L^{2} \oplus L^{3} )$ \\
  $2$ & $(12, 2)$ & $\Delta_{17} \otimes \Delta_{11} \oplus [1]$ & $S_\ell \langle 12 \rangle \otimes L^{3}$ \\
  $3$ & $(11, 1, 0)$ & $\Delta_{17} \otimes \Delta_{11} \oplus [3]$ & $S_\ell \langle 12 \rangle \otimes ( L^{3} \oplus L^{4} )$ \\
  $4$ & $(10, 0, 0, 0)$ & $\Delta_{17} \otimes \Delta_{11} \oplus [5]$ & $S_\ell \langle 12 \rangle \otimes ( L^{3} \oplus L^{4} \oplus L^{5} \oplus L^{6} )$ \\
  $2$ & $(14, 0)$ & $\Delta_{17} \otimes \Delta_{15} \oplus [1]$ & $S_\ell \langle 16 \rangle \otimes L^{1}$ \\
  $2$ & $(13, 3)$ & $\Delta_{19} \otimes \Delta_{11} \oplus [1]$ & $S_\ell \langle 12 \rangle \otimes L^{4}$ \\
  $3$ & $(12, 2, 0)$ & $\Delta_{19} \otimes \Delta_{11} \oplus [3]$ & $S_\ell \langle 12 \rangle \otimes ( L^{4} \oplus L^{5} )$ \\
  $4$ & $(11, 1, 0, 0)$ & $\Delta_{19} \otimes \Delta_{11} \oplus [5]$ & $S_\ell \langle 12 \rangle \otimes ( L^{4} \oplus L^{5} \oplus L^{6} )$ \\
  $5$ & $(10, 0, 0, 0, 0)$ & $\Delta_{19} \otimes \Delta_{11} \oplus [7]$ & $S_\ell \langle 12 \rangle \otimes ( L^{4} \oplus L^{5} \oplus L^{6} )$ \\
  $2$ & $(15, 1)$ & $\Delta_{19} \otimes \Delta_{15} \oplus [1]$ & $S_\ell \langle 16 \rangle \otimes L^{2}$ \\
  $3$ & $(14, 0, 0)$ & $\Delta_{19} \otimes \Delta_{15} \oplus [3]$ & $S_\ell \langle 16 \rangle \otimes ( L^{2} \oplus L^{3} )$ \\
  $2$ & $(16, 0)$ & $\Delta_{19} \otimes \Delta_{17} \oplus [1]$ & $S_\ell \langle 18 \rangle \otimes L^{1}$ \\
  $2$ & $(14, 4)$ & $\Delta_{21} \otimes \Delta_{11} \oplus [1]$ & $S_\ell \langle 12 \rangle \otimes L^{5}$ \\
  $3$ & $(13, 3, 0)$ & $\Delta_{21} \otimes \Delta_{11} \oplus [3]$ & $S_\ell \langle 12 \rangle \otimes ( L^{5} \oplus L^{6} )$ \\
  $4$ & $(12, 2, 0, 0)$ & $\Delta_{21} \otimes \Delta_{11} \oplus [5]$ & $S_\ell \langle 12 \rangle \otimes ( L^{5} \oplus L^{6} )$ \\
  $5$ & $(11, 1, 0, 0, 0)$ & $\Delta_{21} \otimes \Delta_{11} \oplus [7]$ & $S_\ell \langle 12 \rangle \otimes ( L^{5} \oplus L^{6} )$ \\
  $6$ & $(10, 0, 0, 0, 0, 0)$ & $\Delta_{21} \otimes \Delta_{11} \oplus [9]$ & $S_\ell \langle 12 \rangle \otimes ( L^{5} \oplus L^{6} )$ \\
  $2$ & $(16, 2)$ & $\Delta_{21} \otimes \Delta_{15} \oplus [1]$ & $S_\ell \langle 16 \rangle \otimes L^{3}$ \\
  $3$ & $(15, 1, 0)$ & $\Delta_{21} \otimes \Delta_{15} \oplus [3]$ & $S_\ell \langle 16 \rangle \otimes ( L^{3} \oplus L^{4} )$ \\
  $4$ & $(14, 0, 0, 0)$ & $\Delta_{21} \otimes \Delta_{15} \oplus [5]$ & $S_\ell \langle 16 \rangle \otimes ( L^{3} \oplus L^{4} )$ \\
  $2$ & $(17, 1)$ & $\Delta_{21} \otimes \Delta_{17} \oplus [1]$ & $S_\ell \langle 18 \rangle \otimes L^{2}$ \\
  $3$ & $(16, 0, 0)$ & $\Delta_{21} \otimes \Delta_{17} \oplus [3]$ & $S_\ell \langle 18 \rangle \otimes ( L^{2} \oplus L^{3} )$ \\
  $2$ & $(18, 0)$ & $\Delta_{21} \otimes \Delta_{19} \oplus [1]$ & $S_\ell \langle 20 \rangle \otimes L^{1}$ \\
  $2$ & $(15, 5)$ & $\Delta_{23} \otimes \Delta_{11} \oplus [1]$ & $S_\ell \langle 12 \rangle^{\oplus 2} \otimes L^{6}$ \\
  $3$ & $(14, 4, 0)$ & $\Delta_{23} \otimes \Delta_{11} \oplus [3]$ & $S_\ell \langle 12 \rangle^{\oplus 2} \otimes L^{6}$ \\
  $4$ & $(13, 3, 0, 0)$ & $\Delta_{23} \otimes \Delta_{11} \oplus [5]$ & $S_\ell \langle 12 \rangle^{\oplus 2} \otimes L^{6}$ \\
  $5$ & $(12, 2, 0, 0, 0)$ & $\Delta_{23} \otimes \Delta_{11} \oplus [7]$ & $S_\ell \langle 12 \rangle^{\oplus 2} \otimes L^{6}$ \\
  $6$ & $(11, 1, 0, 0, 0, 0)$ & $\Delta_{23} \otimes \Delta_{11} \oplus [9]$ & $S_\ell \langle 12 \rangle^{\oplus 2} \otimes L^{6}$ \\
  $7$ & $(10, 0, 0, 0, 0, 0, 0)$ & $\Delta_{23} \otimes \Delta_{11} \oplus [11]$ & $S_\ell \langle 12 \rangle^{\oplus 2} \otimes L^{6}$ \\
  $2$ & $(17, 3)$ & $\Delta_{23} \otimes \Delta_{15} \oplus [1]$ & $S_\ell \langle 16 \rangle^{\oplus 2} \otimes L^{4}$ \\
  $3$ & $(16, 2, 0)$ & $\Delta_{23} \otimes \Delta_{15} \oplus [3]$ & $S_\ell \langle 16 \rangle^{\oplus 2} \otimes L^{4}$ \\
  $4$ & $(15, 1, 0, 0)$ & $\Delta_{23} \otimes \Delta_{15} \oplus [5]$ & $S_\ell \langle 16 \rangle^{\oplus 2} \otimes L^{4}$ \\
  $5$ & $(14, 0, 0, 0, 0)$ & $\Delta_{23} \otimes \Delta_{15} \oplus [7]$ & $S_\ell \langle 16 \rangle^{\oplus 2} \otimes L^{4}$ \\
  $2$ & $(18, 2)$ & $\Delta_{23} \otimes \Delta_{17} \oplus [1]$ & $S_\ell \langle 18 \rangle^{\oplus 2} \otimes L^{3}$ \\
  $3$ & $(17, 1, 0)$ & $\Delta_{23} \otimes \Delta_{17} \oplus [3]$ & $S_\ell \langle 18 \rangle^{\oplus 2} \otimes L^{3}$ \\
  $4$ & $(16, 0, 0, 0)$ & $\Delta_{23} \otimes \Delta_{17} \oplus [5]$ & $S_\ell \langle 18 \rangle^{\oplus 2} \otimes L^{3}$ \\
  $2$ & $(19, 1)$ & $\Delta_{23} \otimes \Delta_{19} \oplus [1]$ & $S_\ell \langle 20 \rangle^{\oplus 2} \otimes L^{2}$ \\
  $3$ & $(18, 0, 0)$ & $\Delta_{23} \otimes \Delta_{19} \oplus [3]$ & $S_\ell \langle 20 \rangle^{\oplus 2} \otimes L^{2}$ \\
  $2$ & $(20, 0)$ & $\Delta_{23} \otimes \Delta_{21} \oplus [1]$ & $S_\ell \langle 22 \rangle^{\oplus 2} \otimes L^{1}$
    \end{tabular}\vspace{.1in}

    \caption{Types D, E, and F}
    \label{tab:DEF}
\end{table}

\begin{rem} \label{rem:higher_wt}
  One can go much further than $M=23$ in Proposition \ref{pro:params_up_to_23} and Theorem \ref{thm IH tables}.
  Below we give details (but no tables) for $M=24$, sketch the classification for $M=33$, point out a new phenomenon for $M=34$ and explain in what sense these results can be generalized to $M=77$ at least.

  For $M=24$ we find using the same method the following additional $27$ parameters:
  \begin{itemize}
  \item $\Delta_{25}[2] \oplus [2d+1]$ for $0 \leq d \leq 11$, their contribution to intersection cohomology was already computed \ref{it:typeC}.
  \item $\Delta_{15}[2] \oplus \Delta_{11}[2] \oplus [2d+1]$ for $0 \leq d \leq 4$
    Writing $\psi_1 = \Delta_{15}[2]$ and $\psi_2 = \Delta_{11}[2]$ we compute $u_1(\psi) = u_2(\psi) = -1$ and conclude as in \ref{it:typeC} that the contribution to intersection cohomology is
    \[ \Qell(-12) \otimes (\Qell \oplus \Qell(-1))^{\otimes 2} \otimes T_\ell \langle d \rangle. \]
  \item $\Delta_{11}[6] \oplus [2d+1] = \psi_1 \oplus \psi_0$ for $0 \leq d \leq 2$.
    We compute $u_1(\psi) = -\epsilon(\tfrac{1}{2}, \Delta_{11})^{\min(6,2d+1)} = -1$ and
    \[ \spin_{\psi_1}^- \circ \tilde{\alpha}_{\psi_1}|_{\Sp_2 \times \SL_2} \simeq (1 \otimes \Sym^9 \Std_{\SL_2}) \oplus (1 \otimes \Sym^3 \Std_{\SL_2}) \oplus (\Sym^2 \Std_{\Sp_2} \otimes \Sym^5 \Std_{\SL_2}) \]
    (which should be understood as $[10] \oplus [4] \oplus (\Sym^2 \Delta_{11})[6]$) and we deduce that the contribution to intersection cohomology is the sum of
    \[ \left(\bigoplus_{i=12}^{21} \Qell(-i) \oplus \bigoplus_{i=15}^{18} \Qell(-i) \right) \otimes T_\ell \langle d \rangle \ \ \ \text{and} \]
    \[ (\Sym^2 S_\ell \langle 12 \rangle)(-3) \otimes \left( \bigoplus_{i=0}^5 \Qell(-i) \right) \otimes T_\ell \langle d \rangle. \]
  \item $\Delta_{w_1,w_2}[2] \oplus [2d+1] = \psi_1 \oplus \psi_0$ for $(w_1,w_2) \in \{(19,7), (21,5)\}$ for $0 \leq d \leq (w_2-3)/2$.
    We have
    \[ u_1(\psi) = \epsilon(\tfrac{1}{2}, \Delta_{w_1,w_2})^{\min(2,2d+1)} = ((-1)^{1+(w_1+w_2)/2})^{\min(2,2d+1)} = +1 \]
    and we compute
    \[ \spin_{\psi_1}^- \circ \tilde{\alpha}_{\psi_1}|_{\Sp_4 \times \SL_2} \simeq (V_{1,1} \otimes 1) \oplus (1 \otimes \Sym^2 \Std_{\SL_2})\]
    (which should be understood as $(\Lambda^* \Delta_{w_1,w_2})[1] \oplus [3]$) and deduce a contribution to intersection cohomology
    \[ \left( \Qell(-\tfrac{w_1+w_2}{2}-1) \oplus \Qell(-\tfrac{w_1+w_2}{2}+1) \oplus \Lambda^2 S_\ell \langle \tfrac{w_1+w_2-4}{2}, \tfrac{w_1-w_2-2}{2} \rangle \right) \otimes T_\ell \langle d \rangle \]
    (using $\Lambda^2 V_{1,0} \simeq V_{1,1} \oplus V_{0,0}$).
  \item $\psi = \pi[1]$ (for $g=3$ and $\lambda=(9,6,3)$), where $\pi = \Delta_{24,16,8}^o$ is the unique level one self-dual cuspidal automorphic representation for $\PGL_{7,\Q}$ with these weights.
    In this case the contribution is simply (in the notation of \cite[Theorem 5.2.2]{taibi_ecAn})
    \[ \sigma_{\psi,\iota}^\spin(-9) \simeq (\spin_{\psi} \circ \rho_{\psi,\iota}^{\GSpin})(-9) \]
    but thanks to results of Chenevier-Gan \cite{cg} we can say more.
    There is a unique level one cuspidal automorphic representation $\Pi$ for $\GSp_{6,\Q}$ of Siegel type (i.e.\ corresponding to a Siegel eigenform of level $\Sp_6(\Z)$) whose standard Arthur-Langlands parameter is $\psi$.
    By Theorem 3.8, Theorem 3.10 and Proposition 7.11 of \cite{cg} there exists a (unique) Arthur-Langlands parameter $\psi(\pi, \spin)$, which is either of the form $\pi'[1]$ (where $\pi'$ is a level one self-dual cuspidal automorphic representation for $\PGL_{8,\Q}$ with infinitesimal character $(\pm 12, \pm 8, \pm 4, 0, 0)$) or equal to $\pi \oplus [1]$, such that $\sigma_{\psi,\iota}^\spin$ is up to a Tate twist the standard Galois representation associated to $\psi'$.
    Looking at the tables we see that there is no such representation $\pi'$, so we have $\psi(\pi, \spin) = \pi[1] \oplus [1]$ and $\pi$ is of type $G_2$.
    We conclude that the contribution of $\psi$ to intersection cohomology is
    \[ (\std \circ \rho_{\Delta^o_{24,16,8}})(-12) \oplus \Qell(-12). \]
  \item $\psi = \pi[1]$ where $\pi = \Delta_{28,14,6}^o$ (for $g=3$ and $\lambda=(11,5,2)$).
    This case is similar to the previous one except that now the infinitesimal character of $\psi(\pi, \spin)$ is $(\pm 12, \pm 9, \pm 5, \pm 2)$ so we cannot have $\psi(\pi, \spin) = \psi \oplus [1]$ (i.e.\ $\Delta^o_{28,14,6}$ cannot be of type $G_2$) and we conclude $\psi' = \Delta^e_{24,18,10,4}[1]$ (the unique level one self-dual cuspidal automorphic representation for $\PGL_{8,\Q}$ with this infinitesimal character) and $\psi$ contributes
    \[ (\std \circ \rho_{\Delta^e_{24,18,10,4}}^\mathrm{O})(-12) \]
    to intersection cohomology.
  \end{itemize}

  For $M=33$ the only new shapes\footnote{For a formal definition, let the shape of $\pi_i[d_i]$ be $(n_i, d_i)$.} for $\psi_i$ are $\Delta_w[6]$ and $\Delta_w[8]$ (only for $w=11$) and $\Delta_{w_1,w_2,w_3,w_4}^e[1]$ (where $w_1 > w_2 > w_3 > w_4 > 0$ are even integers).
  For $\psi_i = \Delta_{11}[6]$ or $\Delta_{11}[8]$ one can easily express the representations $\sigma_{\psi_i,\iota}^{\spin,\pm}$ occurring in \cite[Theorem 7.1.3]{taibi_ecAn} as direct sums of Tate twists of $\Sym^d S_\ell \langle 12 \rangle$ for $0 \leq d \leq 4$.
  For $\psi_i = \pi[1] = \Delta_{w_1,w_2,w_3,w_4}^{e,(j)}[1]$ we have $k(\psi_i) \in \{30,32,34\}$ and one can show using triality as in \cite{cg} that the representations $\sigma_{\psi_i,\iota}^{\spin,\pm}$ can again be expressed as direct sums of Galois representations associated to self-dual algebraic regular cuspidal representations.
  More precisely, there are two $8$-dimensional orthogonal parameters $\psi(\pi, \spin^\pm)$ having weights
  \[ \pm (\tfrac{w_1+w_2+w_3+w_4}{4}, \tfrac{w_1+w_2-w_3-w_4}{4}, \tfrac{w_1-w_2+w_3-w_4}{4}, \tfrac{w_1-w_2-w_3+w_4}{4}) \text{ (for } \psi(\pi, \spin^+) \text{)} \]
  \[ \pm (\tfrac{w_1+w_2+w_3-w_4}{4}, \tfrac{w_1+w_2-w_3+w_4}{4}, \tfrac{w_1-w_2+w_3+w_4}{4}, \tfrac{w_1-w_2-w_3-w_4}{4}) \text{ (for } \psi(\pi, \spin^-) \text{)} \]
  and such that $\sigma_{\psi_i,\iota}^{\spin, \pm}$ is up to a Tate twist the standard Galois representation associated to $\psi(\pi, \spin^\pm)$.
  Three possibilities can occur.
  \begin{itemize}
  \item If $\psi_i$ is $\psi(\pi', \spin)$ for some $\pi' = \Delta_{w_1',w_2',w_3'}^{o,(j)}$, necessarily with
    \[ \{ \pm w_1, \pm w_2, \pm w_3, \pm w_4\} = \{ \pm \tfrac{w_1'+w_2'+w_3'}{2}, \pm \tfrac{w_1'+w_2'-w_3'}{2}, \pm \tfrac{w_1'-w_2'+w_3'}{2}, \pm \tfrac{w_1'-w_2'-w_3'}{2} \} \]
    and $\pi'$ not of type $G_2$ then
    \[ \{ \psi(\pi,\spin^+), \psi(\pi, \spin^-) \} = \{ \pi'[1] \oplus [1], \pi[1] \}, \]
    distinguished by the presence of $0$ as a weight.
  \item If $\pi = \Delta_{w_1'}^{(j')} \otimes \Delta_{w_1'',w_2''}^{(j'')}$ (tensor product functoriality, see \cite[Theorem 5.3]{cg}) then we have
    \[ \{ \psi(\pi,\spin^+), \psi(\pi, \spin^-) \} = \{ (\Sym^2 \Delta_{w_1'}^{(j')})[1] \oplus (\Lambda^* \Delta_{w_1'',w_2''}^{(j'')})[1], \pi[1] \}, \]
    distinguished by the presence of $0$ as a weight.
  \item Otherwise we have $\psi(\pi,\spin^+) = \pi'[1]$ and $\psi(\pi, \spin^-) = \pi''[1]$ where $\pi'$ and $\pi''$ are both cuspidal.
    Note that in this case it can happen that exactly one of $\pi'$ and $\pi''$ has infinitesimal character at the real place $(\pm w_1', \pm w_2', \pm w_3', 0, 0)$ where $w_1'>w_2'>w_3'>0$ are integers, in which case the associated Galois representation is obtained by $\ell$-adic interpolation from the regular case as in \cite[Corollary 6.4.5]{taibi_ecAn}.
    In this case we have local-global compatibility at all primes $p \neq \ell$ but compatibility at $\ell$ is not yet completely known.
    More precisely we know that the representation is crystalline (as a consequence of \cite[Corollary 6.5.4]{taibi_ecAn}) with Hodge-Tate weights equal to the eigenvalues of the infinitesimal character (by $\ell$-adic interpolation) but not that the eigenvalues of $\Frob_\ell$ on the associated Weil-Deligne representation are those of the Satake parameter $\pm \iota(c(\pi'_\ell))$ or $\pm \iota(c(\pi''_\ell))$, but we cannot yet resolve the $\pm$ ambiguity.
  \end{itemize}

  For $M=34$ a new shape occurs: $\Delta^o_{w_1,w_2,w_3,w_4}[1]$ for
  \[ (w_1,w_2,w_3,w_4) \in \{(30,20,10,8),\, (32,16,14,6), \, (36,16,10,6)\} \]
  and we do not know how to express the associated spin Galois representations $\sigma_{\psi_i,\iota}^\spin$ as direct sums of Galois representations associated to self-dual cuspidal representations.

  We have explicit formulas for numbers of self-dual algebraic regular cuspidal automorphic representations for $\PGL_{n,\Q}$ for $n \leq 24$ (see \cite{taibi_massLF}), so in principle the same method as the one used above allows one to compute the semi-simplification of $\IH^k(A_g^\Sat, \V_\lambda)$ at least for $k + |\lambda| \leq 77 = 12 \cdot 13 / 2 - 1$, in terms of the spin Galois representations $\sigma_{\psi_0,\iota}^\spin$ and $\sigma_{\psi_i,\iota}^{\spin,\pm}$ \cite[Theorem 7.1.3]{taibi_ecAn}.
  For example the dimensions of these intersection cohomology groups can be computed explicitly.
  At present we do not know that the Hodge-Tate weights of the Galois representations $\sigma_{\psi_i,\iota}^{\spin,u_i(\psi)}$ are given by the recipe $\spin_{\psi_i}^{u_i(\psi)}(\tau_{\psi_i})$ (defined in \cite[Definitions 3.1.7 and 4.7.1]{taibi_ecAn}): see Remark \ref{rem:comp_Arthur_Kottwitz}.
  It is however possible to compute the real Hodge structure on $\IH^\bullet(A_g^\Sat, \V_\lambda)$ in the setting of mixed Hodge modules, as explained in Section \ref{sec:L2}.
\end{rem}

\section{Toroidal compactifications of \texorpdfstring{$A_g$}{Ag} and \texorpdfstring{$X_{g,s}$}{Xgs}}\label{sec:compact}
\subsection{Compactifications of \texorpdfstring{$A_g$}{Ag}}\label{sec:Agcompact}

Toroidal compactifications of $A_g$ were constructed over $\Spec \C$ by Ash, Mumford, Rapoport, and Tai \cite{AMRT} and over $\Spec \Z$ by Faltings and Chai \cite{faltingschai}. Let $\Omega^{\rt}_g$ be the cone of symmetric positive semi-definite real $g\times g$ matrices whose kernel has a basis defined over $\Q$. Toroidal compactifications correspond to \emph{admissible decompositions}, which are collections $\Sigma_g$ of rational polyhedral cones lying in $\Omega_g^{\rt}$ satisfying the following properties:
\begin{enumerate}
    \item the union of the cones in $\Sigma_g$ covers $\Omega_g^{\rt}$;
    \item $\Sigma_g$ is closed under taking faces. The intersection of two cones in $\Sigma_g$ is a face of both of them;
    \item $\Sigma_g$ is invariant under the natural action of $\mathrm{GL}_g(\Z)$;
    \item there are finitely many orbits of cones in $\Sigma_g$ under the $\mathrm{GL}_g(\Z)$-action.
\end{enumerate}
An admissible decomposition $\Sigma_g$ of $\Omega_g^{\rt}$ induces an admissible decomposition $\Sigma_{g'}$ of $\Omega_{g'}^{\rt}$ for any $0\leq g'\leq g$. For ease of notation, we denote each of these admissible decompositions by $\Sigma$, leaving the dependence on $g$ implicit.

For every admissible decomposition $\Sigma$ as above, there is a proper stack $\overline A_g^\Sigma$ compactifying $A_g$. For ease of notation we denote it $\overline A_g$, keeping in mind that it depends on a choice of admissible decomposition as above.

If $\Sigma$ is simplicial, meaning every cone $\sigma\in \Sigma$ is simplicial, then $\overline{A}_g$ is smooth as a stack. Any admissible decomposition $\Sigma$ can be refined to a simplicial one. For any $\Sigma$, the boundary of $\overline{A}_g$ is a normal crossings divisor, possibly with self-intersections.

The \emph{dimension} of a cone $\sigma\in \Sigma$ is the dimension of its span in the vector space of quadratic forms on $\R^g$. We say a cone $\sigma\in \Sigma$ is of \emph{rank} $r$ if its generic element corresponds to a symmetric bilinear form of rank $r$.

The toroidal compactifications and the Satake compactification are related as follows. For all admissible decompositions $\Sigma$ and associated compactifications $\Abar_g$, there is a unique map
\[
\varphi:\overline{A}_g\rightarrow A_g^{\Sat}
\]
extending the identity on the interior $A_g$. The space $\overline A_g$ admits a stratification whose set of strata is indexed by the set of cones $\sigma$ in the decomposition $\Sigma$. We write $\beta(\sigma)$ for the locally closed stratum in $\overline A_g$ associated with the cone $\sigma$. Suppose that $\sigma$ has rank $r$ and dimension $d$. Then the restriction of $\varphi$ to $\beta(\sigma)$ has image in the locally closed substack $A_{g-r}$ of $A_g^\Sat$. 

We give a more explicit description of the stratum $\beta(\sigma)$. Let $X_{g-r,r}$ be the $r$-fold fiber product of the universal abelian $(g-r)$-fold over $A_{g-r}$. Over $X_{g-r,r}$, there is a torus bundle $T(\sigma)$ associated to the cone $\sigma$ whose fibers are the quotients $T_r/T_{\sigma}$, where  $T_r = \Sym^2(\Z^{r})\otimes \G_m$ and $T_{\sigma} = (\Sym^2 (\Z^{r})\cap \langle \sigma \rangle )\otimes \G_m$. Consider the finite group $G(\sigma)$, which is the stabilizer of the cone $\sigma$ in $\GL_r(\Z)$. The group $G(\sigma)$ acts on 
\[
T(\sigma)\rightarrow X_{g-r,r}
\]
and $\beta(\sigma) = T(\sigma)/G(\sigma)$.

\subsection{Compactifications of \texorpdfstring{$X_{g,s}$}{Xg,s}}\label{subsec:stratification of Xgs}
Let $\pi:X_g\rightarrow A_g$ be the universal abelian variety, and let $\pi^s:X_{g,s}\rightarrow A_g$ be its $s$-fold fiber product. Faltings and Chai \cite{faltingschai} constructed toroidal compactifications of $X_{g,s}$, again associated to certain combinatorial data. Set
\[
{\Omega}^{\rt}_{g,s} = \{(Q,L_1,\dots,L_s)\in \Omega_g^{\rt}\times (\R^g)^s: L_i \text{ vanishes on the kernel of } Q \text{ for all } i\}.
\]
Here, we view $L_i$ as a row vector, or equivalently a linear form. An \emph{admissible decomposition of ${\Omega}^{\rt}_{g,s}$ relative to an admissible decomposition $\Sigma$ of $\Omega_{g}^{\rt}$} is a collection $\tilde{\Sigma}$ of rational polyhedral cones lying in ${\Omega}^{\rt}_{g,s}$ satisfying the following properties:
\begin{enumerate}
    \item the union of cones in $\tilde{\Sigma}$ covers ${\Omega}^{\rt}_{g,s}$;
    \item $\tilde{\Sigma}$ is closed under taking faces. The intersection of two cones in $\tilde{\Sigma}$ is a face of both of them;
    \item $\tilde{\Sigma}$ is invariant under the natural action of $\GL_g(\Z^g)\ltimes (\Z^g)^s$;
    \item there are finitely many orbits of cones in $\tilde{\Sigma}$ under the $\GL_g(\Z^g)\ltimes (\Z^g)^s$-action;
    \item any cone $\tau\in\tilde{\Sigma}$ maps into a cone $\sigma$ in $\Sigma$ under the the natural projection map.
\end{enumerate}
\noindent

Associated to each such $\tilde{\Sigma}$ is a toroidal compactification of $X_{g,s}$. We denote it $\Xbar_{g,s}$, continuing to suppress the dependence on the decomposition $\tilde\Sigma$ from the notation. The toroidal compactification comes with a map $\overline{\pi^s}:\Xbar_{g,s}\rightarrow \Abar_g$ extending $\pi^s$. Aspects of the geometry of $\Xbar_{g,s}$ are encoded in the combinatorics of $\tilde{\Sigma}$. In particular, if $\tilde{\Sigma}$ is simplicial, then $\Xbar_{g,s}$ is nonsingular. The \emph{dimension} of a cone $\tau\in \tilde{\Sigma}$ is the dimension of its span. We say a cone $\tau\in \tilde{\Sigma}$ is of \emph{rank} $r$ if for a generic element $(Q,L_1,\dots,L_s)$, $Q$ is of rank $r$.

The compactification $\Xbar_{g,s}$ admits a stratification, indexed by the set of cones $\tau$ of $\tilde \Sigma$. Let $\gamma(\tau)$ be the locally closed stratum of $\Xbar_{g,s}$ associated with $\tau$. The stratum $\gamma(\tau)$ has an explicit description as a finite group quotient of a torus bundle. Let $X_{g-r,r}\times_{A_{g-r}} X_{g-r,s}\rightarrow A_{g-r}$ be the $(r+s)$-fold fiber product of the universal abelian $(g-r)$-fold. There is a torus bundle $T(\tau)\rightarrow X_{g-r,r}\times_{A_{g-r}} X_{g-r,s}$ whose fibers are given by $(\Sym^2(\Z^r)\times (\Z^r)^s)\otimes \G_m/(\Sym^2(\Z^r)\times (\Z^r)^s\cap \langle \tau\rangle)\otimes \G_m$. The finite group $G(\tau)$, which is the stabilizer of the cone $\tau$ in $\GL_{r}(\Z)\ltimes (\Z^r)^s$ acts on $T(\tau)\rightarrow X_{g-r,r}\times_{A_{g-r}} X_{g-r,s}$, and $\gamma(\tau)=T(\tau)/G(\tau)$.  Note that in the case that $s = 0$, we recover exactly the stratification of the compactifications of $A_g$ as in Section \ref{sec:Agcompact}. 

The stratifications of $\Xbar_{g,s}$ and $\overline A_g$ are compatible in the following sense. If $\tau$ maps into $\sigma$ under the natural projection map, then $\gamma(\tau)$ maps to $\beta(\sigma)$ under $\overline{\pi^s}$.

\section{Cohomology of the strata}\label{sec:stratacoh}
\subsection{Fibered powers of the universal abelian variety}

Recall $\pi^s: X_{g,s} \rightarrow A_g$ is the $s$-fold fiber product of the universal abelian $g$-fold $\pi: X_g\rightarrow A_g$.
\begin{lem}\label{lem:decomposition}
    The local system $R^q\pi^s_*\Q_{\ell}(-k)$ splits as a direct sum of $\V_{\lambda}(-m)$ such that $m \geq k$, $\lambda_1 \leq s$, and $|\lambda|+2m = q+2k$. 
\end{lem}
\begin{proof}It suffices to prove this when $k=0$, since the functor $R^q\pi^s_*$ commutes with Tate twist. By the K\"unneth formula, we have
    \[
    R^q\pi^s_*\Q_{\ell} \cong \bigoplus_{i_1+\dots+i_s = q} R^{i_1}\pi_*\Q_{\ell} \otimes \dots \otimes R^{i_s} \pi_* \Q_{\ell},
    \]
    where $\pi^1=\pi$.
    We have $R^1\pi_*\Q_{\ell} \cong \V_1$, and also $R^{i}\pi_* \Q_{\ell} \cong \bigwedge^{i} \V_1.$
    For $0 \leq i \leq g$ we have
    \[ \bigwedge^i \V_1 \simeq \bigoplus_{0 \leq j \leq \lfloor i/2 \rfloor} \V_{1^{i-2j}}(-j) \]
    and
    \[ \bigwedge^{2g-i} \V_1 \simeq \bigwedge^i \V_1 (i-g) \simeq \bigoplus_{0 \leq j \leq \lfloor i/2 \rfloor} \V_{1^{i-2j}}(i-g-j). \]
    
    Considering weights for the diagonal torus of $\mathrm{Sp}_{2g}$ we see that the representation
    \[ V_{1^{i_1-2j_1}} \otimes \dots \otimes V_{1^{i_s-2j_s}} \]
    of $\mathrm{Sp}_{2g}$ decomposes as a direct sum of $V_\lambda$'s satisfying $\sum_{k=1}^s 1^{i_k-2j_k} - \lambda = \sum_{\alpha \in \Delta} n_\alpha \alpha$ where $n_\alpha \geq 0$ and
    \[ \Delta = \{(1,-1,0,\dots,0), \dots, (0,\dots,0,1,-1), (0,\dots,0,2)\} \]
    is the set of simple roots.
    In particular $V_\lambda$ occurs only if $\lambda_1 \leq s$.
    The condition $|\lambda|+2m=q$ follows from consideration of weights (equivalently, central characters for representations of $\mathrm{GSp}_{2g}$).
\end{proof}
\subsection{Spectral sequence of a stratified space}
Let $Y$ be a Deligne--Mumford stack over a field. Suppose $Y$ admits a filtration by closed substacks
\[
\emptyset = Y_{-1} \subset Y_0 \subset \dots \subset Y_n = Y.
\]
Then there is a spectral sequence
\begin{equation}\label{BMSS}
E_1^{p,q} = H^{p+q}_{c}(Y_p\smallsetminus Y_{p-1},\Q_{\ell})\implies H^{p+q}_{c}(Y,\Q_{\ell})
\end{equation}
 The differentials are compatible with the weight filtration. In particular, if $Y$ is smooth and proper, only the pure weight part of each entry $E_1^{p,q}$ can survive to $E_{\infty}^{p,q}$.  
 
The construction is well known to experts, but we explain it for the reader's convenience. 
For each $p$ let $j(p)$ be the open embedding $(Y\smallsetminus Y_p) \hookrightarrow Y$, and let $i(p)$ be the locally closed embedding $(Y_p \smallsetminus Y_{p-1}) \hookrightarrow Y$. For each $p$ there is a short exact sequence of constructible sheaves on $Y$,
\[ 0 \to j(p+1)_!\Q_\ell \to j(p)_!\Q_\ell \to i(p+1)_!\Q_\ell \to 0.\]
We now consider the chain of monomorphisms
$$ 0 = j(n)_!\Q_\ell \to j(n-1)_! \Q_\ell \to \dots \to j(1)_!\Q_\ell \to j(0)_!\Q_\ell \to j(-1)_!\Q_\ell = \Q_\ell. $$
It may be considered as a decreasing filtration of $\Q_\ell$ with $F^p \Q_\ell = j(p)_!\Q_\ell$, and with associated graded pieces given by $\gr_F^p \Q_\ell = F^p\Q_\ell/F^{p+1}\Q_\ell = i(p+1)_!\Q_\ell$. There is an associated filtration of $R\Gamma_c(Y,\Q_\ell)$, and an associated cohomology spectral sequence, with $E_1$-page given by
$$E_1^{pq} = H^{p+q}_c(Y,\gr_F^p \Q_\ell) = H^{p+q}_c(Y,i(p+1)_!\Q_\ell) = H^{p+q}_c(Y_{p+1} \smallsetminus Y_{p},\Q_\ell) $$
converging to $H^\bullet_c(Y,\Q_\ell)$; this gives our spectral sequence after reindexing.   
\subsection{Subquotients of the cohomology of toroidal compactifications of \texorpdfstring{$X_{g,s}$}{Xgs}}
Let $\Sigma$ be a simplicial admissible decomposition of $\Omega_g^{\rt}$ and $\tilde{\Sigma}$ a simplicial admissible decomposition of ${\Omega}_{g,s}^{\rt}$ over $\Sigma$. We denote by $\overline{\pi^s}:\Xbar_{g,s}\rightarrow \Abar_{g}$ the associated morphism between the moduli stacks, both of which are nonsingular and proper. 
\begin{lem}\label{stratification lemma 1}
    The semisimplification of $H^i(\overline{X}_{g,s})$ is a direct sum of subquotients of $\bigoplus_{\tau\in \tilde{\Sigma} }\gr^W_{i}H^i_{c}(\gamma(\tau))$.
\end{lem}
\begin{proof}
The stratification described in \cref{subsec:stratification of Xgs} induces a filtration of $Y=\overline{X}_{g,s}$, with $Y_p$ the union of all strata $\gamma(\tau)$ of dimension at most $p$. The associated spectral sequence \eqref{BMSS} reads 
\begin{equation*}\label{spectralsequence2}
     \bigoplus_{\dim \gamma(\tau)=p} H^{p+q}_{c}(\gamma(\tau))\implies H^{p+q}_{c}(\overline{X}_{g,s}) = H^{p+q}(\overline{X}_{g,s}).
 \end{equation*}
Because $H^{p+q}(\overline{X}_{g,s})$ is pure, only the pure weight part of $H^{p+q}_{c}(\gamma(\tau))$ can survive to the $E_\infty$ page. Thus, the semisimplification of $H^i(\overline{X}_{g,s})$ is a direct sum of subquotients of $\bigoplus_{\tau}\gr^W_{i} H_{c}^i(\gamma(\tau))$. 

\end{proof}

\begin{lem} \label{stratification lemma 2}
  Let $\tau \in \tilde\Sigma$ be a cone of rank $r$ and dimension $d$.
  Let $k=\binom{r+1}{2}+rs-d$ be the rank of the associated torus bundle $T(\tau) \to X_{g-r,s+r}$ described in \cref{subsec:stratification of Xgs}.
  The semisimplification of $\gr^W_i H^i_c(\gamma(\tau))$ is a direct sum of subquotients of pure weight cohomology groups $\gr^W_{i} H_c^p(A_{g-r},\V_{\lambda}(-m))$, where $(\lambda,m)$ ranges over dominant weights of $\mathrm{GSp}_{2(g-r)}$ satisfying $p+|\lambda|+2m=i$, $m\geq k$, and $\lambda_1 \leq s+r$.
\end{lem}

\begin{proof}
 The torus bundle $T(\tau)$ is an open substack of an affine bundle over $X_{g-r,s+r}$ of rank $k$. Indeed, the torus bundle is a $k$-fold product of $\G_m$-bundles, each of which is the complement of the zero section of some line bundle $L_j$ over $X_{g-r,s+r}$. Therefore, $T(\tau)$ is the complement of the zero section of the direct sum $V:=\bigoplus_{j=1}^k L_j$. From the long exact sequence in compactly supported cohomology, we have an injection
 \[\gr^W_i H^i_c(T(\tau))\hookrightarrow \gr^W_i H^i_c(V).
 \] 
 Moreover, we have an injection
 \[
 \gr^W_i H^i_c(\gamma(\tau))\cong \gr^W_i H^i_c(T(\tau))^{G(\tau)}\hookrightarrow \gr^W_i H^i_c(T(\tau)).
 \]
 Because $V$ is a vector bundle over $X_{g-r,s+r}$ of rank $k$, we have
 \[H^i_c(V,\Q_{\ell}) \cong H^{i-2k}_c(X_{g-r,s+r},\Q_{\ell}(-k)).
 \]
Composing, we have an injection
\[
\gr_i^W H^i_c(\gamma(\tau))\hookrightarrow  \gr^W_{i} H^{i-2k}_c(X_{g-r,s+r},\Q_{\ell}(-k)) .
\]
To compute $\gr^W_{i} H^{i-2k}_c(X_{g-r,s+r},\Q_{\ell}(-k))$, we apply the Leray spectral sequence for compactly supported $\ell$-adic cohomology for the projection map $\pi^{s+r}:X_{g-r,s+r}\rightarrow A_{g-r}$. The Leray spectral sequence
\begin{equation}\label{Leray}
            E_2^{p,q} = H^p_c(A_{g-r},R^q\pi^{s+r}_*\Q_{\ell}(-k))\implies H_c^{p+q}(X_{g-r,s+r},\Q(-k))
\end{equation}
is compatible with weights, so the semisimplification of $\gr^W_{i}H^{i-2k}_c(X_{g-r,s+r},\Q_{\ell}(-k))$ is a direct sum of subquotients of $
\gr^W_{i}H^p_c(A_{g-r},R^q\pi^{s+r}_*\Q_{\ell}(-k))$, for $p+q=i-2k$. Applying Lemma \ref{lem:decomposition} yields the result.
\end{proof}

	\begin{prop}\label{prop:allsubquotients} The semisimplification of $H^i(\overline X_{g,s})$ is a direct sum of subquotients of intersection cohomology groups $\IH^p(A^\Sat_{g-r},\V_{\lambda}(-m))$, where $g\geq r \geq 0$, and $(\lambda,m)$ ranges over dominant weights of $\mathrm{GSp}_{2(g-r)}$ satisfying $p+|\lambda|+2m=i$, and $m\geq \binom{r+1} 2 + rs$, and $\lambda_1 \leq s+r$.
	\end{prop}

    \begin{proof}
        This follows from \cref{stratification lemma 1} and \cref{stratification lemma 2}, since each  $\gr^W_{i} H_c^p(A_{g-r},\V_{\lambda}(-m))$ occurring in \cref{stratification lemma 2} is a subquotient of the corresponding intersection cohomology group $\IH^p(A^\Sat_{g-r},\V_{\lambda}(-m))$. Indeed, it follows from the argument of \cite[Lemma 2]{durfee} (but with coefficients taken in a local system) that the natural map $$H_c^p(A_{g-r},\V_{\lambda}(-m)) \to \IH^p(A^\Sat_{g-r},\V_{\lambda}(-m))$$ induces an injection on the pure weight quotient.
    \end{proof}

\begin{proof}[Proof of Theorem \ref{thm:main}]
    The result now follows by combining Proposition \ref{prop:allsubquotients} and Theorem \ref{thm IH}.
\end{proof}

\section{Non-Tate cohomology}\label{sec:nontate}

	\begin{lem}\label{hansen lemma}
		Let $f\colon X \to Y$ be a projective morphism of separated Deligne--Mumford stacks of finite type over a field $k$, with $X$ nonsingular. Let $j: U \hookrightarrow Y$ be a dense open embedding such that $f^{-1}(U) \to U$ is smooth. Then $\oplus_n j_{!\ast}j^\ast R^nf_\ast\Q_\ell[-n]$ is a direct summand of $Rf_\ast\Q_\ell$.
	\end{lem}
	
	\begin{proof}By the decomposition theorem \cite{bbd}, the result holds for $X_{\overline k} \to Y_{\overline k}$, but because the splitting afforded by the decomposition theorem is not canonical, it has no \emph{a priori} reason to descend to $k$.
    But by relative Hard Lefschetz and \cite{delignedegeneration}, and also using that $D^b_c(Y,\Q_\ell)$ is the derived category of perverse sheaves \cite{beilinsonperverse}, we see that $Rf_\ast\Q_\ell \simeq \bigoplus_n {}^\mathfrak{p}H^n(Rf_\ast\Q_\ell)[-n]$ over $k$.
    And by the argument of \cite[Lemma 3.2.6]{hansenzavyalov}, $j_{!\ast} j^\ast{}^\mathfrak{p}H^n (Rf_\ast\Q_\ell)$ is a direct summand of ${}^\mathfrak{p}H^n (Rf_\ast\Q_\ell)$ for all $n$, using that $j^\ast {}^\mathfrak{p}H^n (Rf_\ast\Q_\ell)$ is a shifted local system. The fact that $X$ and $Y$ are stacks plays no role in the argument. Indeed, we may pass to coarse spaces, and since the coarse space $cX$ of $X$ has only finite quotient singularities, it still satisfies $\IC_{cX}=\Q_\ell[\dim cX]$. 
	\end{proof}

Recall the function $c(g)$ defined in Table \ref{nontatetable} and used in the statement of Theorem \ref{thm:tateornot}.
\begin{proof}[Proof of Theorem \ref{thm:tateornot} for $(g,s)\neq (3,6), (3,7)$]  
Suppose $g \geq 4$ and $s \geq c(g)=7-g$ or $g=3$ and $s\geq 8$.
Now apply \cref{hansen lemma} to the map $f:\Xbar_{g,s}\to A_g^\Sat$, with $j : A_g \to A_g^\Sat$. Then 
\[ \IH^{\bullet-q}(A^{\Sat}_g, R^q \pi^s_* {\Qell}) \hookrightarrow H^\bullet(\Xbar_{g,s},\Qell). \]
\begin{itemize}
\item For $g=3$ and $s\geq8$, taking $\lambda=(8,4,4)$ the coefficient system $\V_{\lambda}$ occurs in $R^{16}\pi^s_* {\Qell}$ and the parameter $\Delta_{11}[4] \oplus (\Sym^2\Delta_{11})[1]$ contributes a non-Tate Galois representation to $\IH^\bullet(A^{\Sat}_g,\V_{\lambda})$ (see Table \ref{tab:DEF}).
\item For $4 \leq g \leq 7$ and $s \geq 7-g=c(g)$, the coefficient system $\V_\lambda$ occurs in $R^{4c(g)} \pi^s_* {\Qell}$ with $\lambda = (c(g),c(g),c(g),c(g),0,\dots,0)$ and the parameter $\Delta_{11}[4] \oplus [2g-7]$ contributes a non-Tate Galois representation to $\IH^\bullet(A^{\Sat}_g, \V_\lambda)$ (see Table \ref{tab:DEF}).
\item For $g=8$ we need to show that $\IH^\bullet(A^{\Sat}_8, \Qell)$ is not Tate.
  For this we consider the parameter $\psi = \Delta_{11}[6] \oplus [5]$.
  We have $u_1(\psi) = - \epsilon(\tfrac{1}{2}, \Delta_{11})^{\min(5,6)} = -1$.
  Computing the composite $\spin_{\Delta_{11}[6]}^- \circ \tilde{\alpha}_{\Delta_{11}[6]}$ in \cite[\S 6.3]{taibi_ecAn} using weights we see that $(\Sym^2 S_\ell \langle 12 \rangle)(-3)$ appears in $\IH^{28}(A_8^\Sat, \Qell)$.
\item For $g \geq 9$, we again need to show $\IH^\bullet(A^{\Sat}_g, \Qell)$ is not Tate. We use parameters of the form $\Delta_{2g-3}^{(j)}[4] \oplus [2g-7]$ (note that $2g-3 \geq 15$ so $s\langle{2g-3}\rangle > 0$).
  We have $u_1(\psi) = \epsilon(\tfrac{1}{2}, \Delta_{2g-3}^{(j)})^{\min(4,2g-7)} = 1$ since $2g-7 \geq 11 > 4$, so the computation is identical to the one in \ref{it:typeE}, and we find that $\oplus_j \Sym^2 \rho_{\Delta_{2g-3}^{(j)}}$ appears in $\IH^{4g-6}(A_g^\Sat, \Qell)$.
\end{itemize}

Now suppose $g\geq 4$ and $g + s\leq 6$.
By Proposition \ref{prop:allsubquotients}, it suffices to show that $\IH^\bullet(A^\Sat_{g-r},\V_{\lambda})$ is Tate, for all $g\geq r \geq 0$ and $\lambda$ ranging over dominant weights of $\mathrm{Sp}_{2(g-r)}$ such that $\lambda_1 \leq s+r$.
Because of the inequality $\lambda_1 + g-r \leq 6$ the only contributing parameters are of the form $[2g+1]$ or $\Delta_{11}[2] \oplus [2d+1]$.
In the first case the contribution is clearly Tate, and in the second case as well by the calculation in \ref{it:typeC}. For $g=1$ and $s<10$, examining Tables \ref{tab:rest}, \ref{tab:C}, and \ref{tab:DEF} shows that there are no nontrivial parameters contributing, and so the only contribution is Tate. If $g=2$ and $s<7$, the only contributing parameters are $[5]$, $\Delta_{11}[2]\oplus [1]$, and $\Delta_{15}[2]\oplus [1]$. The first contribution is clearly Tate, and the latter contributions are Tate by the calculations in \ref{it:typeC}. Finally, if $g=3$ and $s<6$, the only contributing parameters are $[7]$, $\Delta_{11}[2]\oplus [3]$, and $\Delta_{15}[2]\oplus [3]$. The first contribution is clearly Tate, and the latter contributions are Tate by the calculations in \ref{it:typeC}.
\end{proof}

It remains to prove Theorem \ref{thm:tateornot} for $(g,s)=(3,6)$ or $(3,7)$. The strategy for showing that $H^{\bullet}(\Xbar_{3,s},\Qell)$ is not Tate is different from the other cases. We will find that  $\IH^\bullet(A_3^{\Sat},\V_{\lambda})$ is Tate for all contributing local systems $\V_{\lambda}$. Instead we shall add up, for every stratum $S$, the compactly supported Euler characteristic $e_c(S)$ taken in the Grothendieck group of $\ell$-adic Galois representations. The non-Tate contributions will all be Tate twists of the Galois representation $S_\ell \langle 18 \rangle$. We will find that these occur in $e_c(S)$ both when $S$ is the interior and when $S$ is a torus rank one stratum, with opposite signs, and we need to verify that the contributions cannot cancel.

\begin{defn}
    Let $M,N$ be elements in the Grothendieck group of $\ell$-adic Galois representations. We write $M \equiv N$ if $M-N$ is a polynomial in $L := \Qell(-1)$.
\end{defn}

\subsection{Interior}

\begin{lem}\label{lem:X36X37}
  We have
  \begin{align} \label{eq:ec X_3_6}
    e_c(X_{3,6}) &\equiv (L^6 + 21L^5 + 120L^4 + 280L^3 + 309L^2 + 161L + 32) \cdot S_\ell \langle 18 \rangle, \\
	\label{eq:ec X_3_7}
	e_c(X_{3,7}) &\equiv  (28 L^9 + 784 L^8 + 6616 L^7 + 25984 L^6 + 57400 L^5 + 77952 L^4 + \\ & \hspace{3cm}+67032 L^3 + 35728 L^2 + 10780 L + 1408) \cdot S_\ell \langle 18 \rangle. \nonumber
\end{align}

\end{lem}
\begin{proof}
We go back to the proof of Lemma \ref{lem:decomposition} and compute the multiplicity of each relevant $V_\lambda$ in
\[ V_{1^{i_1-2j_1}} \otimes \dots \otimes V_{1^{i_s-2j_s}}. \]
More precisely we care about dominant weights $\lambda$ satisfying $\lambda_1 \leq 7$ and such that $e_c(A_3, \V_\lambda)$, which we can compute using \cite[Theorem 8.3.1]{taibi_ecAn}, is not Tate.
This formula expresses $e_c(A_3, \V_\lambda)$ as a linear combination of $e_{\IH}(A_g^\Sat, \V_{\lambda'})$ where $g \leq 3$ and $\lambda_1' + g \leq s+3$.
Denoting $\tau = \lambda' + \rho$ this inequality implies that for any
\[ \psi = \psi_0 \oplus \dots \oplus \psi_r \in \tPsidut(\Sp_{2g}) \]
with $\psi_i = \pi_i[d_i]$ we have $\pi_i \in \{1, \Delta_{11}, \Delta_{15}, \Delta_{17}, \Delta_{19}, \Delta_{19,7}\}$ for all $i$.
Because we have $g \leq 3$ we cannot have $\pi_i = \Delta_{19,7}$ (note that in this case $d_i \geq 2$ as it is even), and similarly we conclude
\[ \psi_i \in \{[1], [3], [5], [7], \Delta_{11}[2], \Delta_{15}[2], \Delta_{17}[2] \}. \]
Thus $\psi$ is either of the form $[2d+1]$ or $\Delta_w[2] \oplus [2d+1]$, and we know \ref{it:typeC} that intersection cohomology is Tate in these cases except for $\psi = \Delta_{17}[2] \oplus [1]$, which occurs for $g=2$ and $\V_{7,7}$.
By enumeration, the relevant dominant weights are just $\lambda = (6,6,\lambda_3)$ for $\lambda_3 \in \{0,2,4,6\}$, for which we compute (details omitted) using \cite[Theorem 8.3.1]{taibi_ecAn}
\[ e_c(A_3, \V_\lambda) \equiv - e_{\IH}(A_2^\Sat, \V_{7,7}) \equiv S_\ell \langle 18 \rangle. \]
We need to consider all tuples of integers $(i_1, \dots, i_s)$ satisfying $0 \leq i_k \leq 3$, all tuples $(j_1,\dots,j_s)$ of integers satisfying $0 \leq 2j_k \leq i_k$, and all subsets $S$ of $\{1,\dots,s\}$ satisfying $i_k \neq 3$ for all $k \in S$, such that $\V_{6,6,\lambda_3}(-m)$ (with $\lambda_3$ even) occurs in
\[ \V_{1^{i_1-2j_1}} \otimes \dots \otimes \V_{1^{i_s-2j_s}} \left( -\sum_{k=1}^s j_k + \sum_{k \in S} i_k-g \right). \]
This implies $\# \{ k \,|\, i_k-2j_k \geq 2 \} \geq 6$ and $q=\sum_{k=1}^s i_k$ even.
\begin{itemize}
\item We first consider the case $s=6$.
  We have $2 \leq i_k \leq 3$ and $j_k=0$ for all $k$.
  Weight computations yield
  \begin{align*}
  \V_{1^3}^{\otimes 6} &\simeq \V_{6,6,6} \oplus \V_{6,6,4}(-1)^{\oplus 5} \oplus \V_{6,6,2}(-2)^{\oplus 9} \oplus \V_{6,6,0}(-3)^{\oplus 5} \oplus \text{irrelevant} \\
  \V_{1^3}^{\otimes 4} \otimes \V_{1^2}^{\otimes 2} &\simeq \V_{6,6,4} \oplus \V_{6,6,2}(-1)^{\oplus 3} \oplus \V_{6,6,0}(-2)^{\oplus 2} \oplus \text{irrelevant} \\
  \V_{1^3}^{\otimes 2} \otimes \V_{1^2}^{\otimes 4} &\simeq \V_{6,6,2} \oplus \V_{6,6,0}(-1) \oplus \text{irrelevant} \\
  \V_{1^2}^{\otimes 6} &\simeq \V_{6,6,0} \oplus \text{irrelevant}. \\
  \end{align*}
  Taking into account the number of times each of these tensor products occurs as well as the possible sets $S$ as above we obtain, writing $1$ for $\Qell$ and $L$ for $\Qell(-1)$
  \begin{align*}
    e_c(X_{3,6}) \equiv &\ e_c(A_3, \V_{6,6,6}) + (15+35L+15L^2) \cdot e_c(A_3, \V_{6,6,4}) \\
    &\ + (15+105L+189L^2+105L^3+15L^4) \cdot e_c(A_3, \V_{6,6,2}) \\
    &\ + (1+21L+105L^2+175L^3+105L^4+21L^5+L^6) \cdot e_c(A_3, \V_{6,6,0}) \\
    \equiv &\ (L^6 + 21L^5 + 120L^4 + 280L^3 + 309L^2 + 161L + 32) \cdot S_\ell \langle 18 \rangle
  \end{align*}
\item Now consider the case $s=7$.
  The computations are similar to the $s=6$, only more tedious, so we had them done by a computer and found \eqref{eq:ec X_3_7}. \qedhere
\end{itemize}

\end{proof}
  \todo[inline]{TODO: check independently}

\subsection{Torus rank \texorpdfstring{$1$}{1} part of the boundary}

 Let $\Xbar_{g,s}$ be a toroidal compactification of $X_{g,s}$. By the torus rank $1$ part of the boundary we mean the union of the strata corresponding to cones of rank $1$. Equivalently, this is the locus in $\Xbar_{g,s}$ over which the universal semiabelian variety (pulled back from $\Abar_g$) has torus rank $1$.
 
 We will need to describe more explicitly the rank $1$ cones in an admissible decomposition of $\Omega_{g,s}^\rt$.
 When $s=0$, it is a direct verification that there is a unique orbit of cones $\sigma$ of rank $1$ in any admissible decomposition of $\Omega_g^\rt$.
 A rank one cone $\sigma$ is one-dimensional, of the form $\R_{\geq 0} q$, and has $G(\sigma) = \mathrm{GL}_1(\Z) = \{\pm 1\}$.
 We fix such a cone $\sigma$.
 The corresponding stratum $\beta(\sigma)$ in $\Abar_g$ is a boundary divisor isomorphic to $[X_{g-1,1}/\{\pm 1\} ]$. When $s>0$, the inverse image of $\sigma$ in $\Omega_{g,s}^\rt$ is isomorphic to $\R_{\geq 0} q \times \R^s$.
 Associating to each rank $1$ cone $\tau$ in $\Omega_{g,s}^\rt$ above $\sigma$ the convex polyhedron $\tau^\aff := \tau \cap \{q\} \times \R^s$ we obtain that the torus rank $1$ part of the boundary is described combinatorially by a polyhedral decomposition of $\R^s$, equivariant for the action of $\Z^s \rtimes \{\pm 1\}$.
 Torus rank $1$ strata inside $\Xbar_{g,s}$ are then in bijection with $(\Z^s \rtimes \{\pm 1\})$-orbits of polyhedra.
 If the associated toroidal compactification is smooth, then we have a triangulation of $\R^s$.

\begin{prop}\label{prop:number of torus rank one strata}Let $\widetilde\Sigma$ be an admissible decomposition of $\Omega^\rt_{g,s}$. In this decomposition, we have:
\begin{enumerate}
    \item If $\tau$ is a rank $1$ cone, then $G(\tau)$ is trivial or cyclic of order $2$.
    \item There are exactly $2^s$ orbits of rank $1$ cones with $G(\tau) \cong \Z/2$.
    If the decomposition is simplicial, then all of them have dimension $1$ or $2$.
\end{enumerate}
    
\end{prop}
\begin{proof}The stabilizer of a rank $r$ cone is a finite subgroup of $(\Z^r)^s \rtimes \mathrm{GL}_r(\Z)$, hence a finite subgroup of $\mathrm{GL}_r(\Z)$. This implies the first part. 

For the second part, we consider the induced cell decomposition of the real torus $(\R/\Z)^s$.
The cell decomposition is equivariant under the action of $\pm 1$, and a cone with $G(\tau) \cong \Z/2$ corresponds to a cell which is fixed by $\pm 1$.
Topologically, a fixed cell is necessarily homeomorphic to $\R^d$ (for $d=\dim \tau^\aff = \dim \tau - 1$), with the action by $\pm 1$. In particular, each fixed cell has a unique fixed point. But the action of $\pm 1$ on an $s$-dimensional real torus has exactly $2^s$ fixed points. 

For the last part, consider an involution of a simplex $\tau^\aff$. We claim that if the involution has isolated fixed points, then the simplex is a vertex or an edge. Indeed, if the simplex has more than two vertices, then the involution acting on the set of vertices is not transitive. The involution fixes at least two distinct points on the boundary (vertices or midpoints of edges), but then also the line between these two points is fixed, a contradiction.
\todo[inline]{O: alternative argument (simpler for me, maybe not for everyone): letting $b$ be the barycenter of $\tau^\aff$ there is a linearly independent family $(e_1,\dots,e_d)$ of $\R^s$ such that the vertices of $\tau^\aff$ are $b+e_1,\dots,b+e_d$ and $b-\sum_i e_i$ and the involution is $b+v \mapsto b-v$, and we see that the set of vertices is not preserved by this involution if $d>1$.}
\end{proof}

\begin{lem}\label{lem: ec of torus rank one strata}
	Let $\gamma(\tau)$ be a torus rank $1$ stratum of $\Xbar_{3,6}$. 
    
    \begin{enumerate}
        \item If $G(\tau)$ is trivial, then $e_c(\gamma(\tau)) \equiv -(L-1)^k \cdot S_\ell\langle 18\rangle$, where $k=6-\dim(\tau^\aff)$ is the rank of the torus bundle described in \cref{subsec:stratification of Xgs}.   
        \item If $G(\tau) \cong \Z/2$ and $\dim \tau^\aff = 0$, then $e_c(\gamma(\tau)) \equiv -(L^6+15L^4+15 L^2+1) \cdot S_\ell\langle 18 \rangle$. 
        \item If $G(\tau) \cong \Z/2$ and $\dim \tau^\aff = 1$, then $e_c(\gamma(\tau)) \equiv -(L^5+10L^3+5L) \cdot S_\ell\langle 18 \rangle$. 
    \end{enumerate}
    Let similarly $\gamma(\tau)$ be a torus rank $1$ stratum of $\Xbar_{3,7}$. 
    \begin{enumerate}
        \item If $G(\tau)$ is trivial, then $e_c(\gamma(\tau)) \equiv -(36+64L+36L^2) (L-1)^k \cdot S_\ell\langle 18\rangle$, where $k=7-\dim(\tau^\aff)$ is the rank of the torus bundle described in \cref{subsec:stratification of Xgs}.   
        \item If $G(\tau) \cong \Z/2$ and $\dim \tau^\aff = 0$, then $$e_c(\gamma(\tau)) \equiv -(29 L^9 + L^8 + 540 L^7 + 756 L^6 + 1134 L^5 + 1358 L^4 + 924 L^3 + 196 L^2 + 189 L - 7) \cdot S_\ell\langle 18 \rangle.$$ 
        \item If $G(\tau) \cong \Z/2$ and $\dim \tau^\aff = 1$, then 
        $$e_c(\gamma(\tau)) \equiv -(29 L^8 + 8 L^7 + 380 L^6 + 568 L^5 + 590 L^4 + 568 L^3 + 380 L^2 + 8 L + 29) \cdot S_\ell\langle 18 \rangle.$$ 
    \end{enumerate}
\end{lem}

\begin{proof}
		Consider first the case that $G(\tau)$ is trivial. Then $\gamma(\tau) = T(\tau)$ is a torus bundle $T(\tau)\stackrel p \rightarrow X_{2,s+1}$ of rank $k$. The local systems $R^qp_!\Qell$ are trivial, 
		because $p$ is a Zariski locally trivial fibration. We get a spectral sequence
		\begin{equation}
			\label{SS} E_2^{pq}=H^p_c(X_{2,s+1}) \otimes H^q_c(\mathbb G_m^k) \implies H^{p+q}_c(T(\tau))
		\end{equation}
		and so $e_c(T(\tau)) = (L-1)^k e_c(X_{2,s+1})$. 
		
		We then study $e_c(X_{2,s+1})$ through the Leray spectral sequence of the fibration $\pi^{s+1}:X_{2,s+1}\rightarrow A_2$. We know the cohomology of any symplectic local system $\V_\lambda$ on $A_2$ \cite{localsystemsA2}. These results show in particular that $\lambda  = (7,7)$ is the unique $\lambda$ for which $H^\bullet_c(A_2,\V_\lambda)$ has non-Tate cohomology and which minimizes $\lambda_1$: the representation $S_\ell\langle 18 \rangle$ appears in $H^3_c(A_2,\V_{7,7})$.  The second smallest value of $\lambda_1$ for which there is a local system with non-Tate cohomology is $\lambda_1=9$. Consequently, for $s \in \{6,7\}$ we have 
		\[
		e_c(X_{2,s+1})\equiv -g_s(L)S_{\ell}\langle 18\rangle
		\]
		for a polynomial $g_s \in \Z[L]$, depending on $s$; namely, the $i$th coefficient of $g$ is the multiplicity of $\V_{7,7}(-i)$ in $R^{14 + 2i}\pi^{s+1}_\ast\Q_\ell$, where $\pi^{s+1} : X_{2,s+1} \to A_2$. (Indeed, by what was said above, all $\V_\lambda$ except $\V_{7,7}$ contribute only Tate classes; since $\V_{7,7}(-i)$ has weight $14+2i$, it can occur nontrivially in $R^j\pi^{s+1}_\ast\Q_\ell$ only for $j=14+2i$.) A direct computation similar to that in Lemma \ref{lem:X36X37} shows that 
		\begin{equation} \label{computation of gs} g_6(L)= 1 \qquad \text{and} \qquad g_7(L) = 36+64L+36L^2.\end{equation}
	This gives the result for $G(\tau)$ trivial.

		We now consider the case that $G:=G(\tau)$ is nontrivial (hence cyclic of order $2$), and $\gamma(\tau) = [T(\tau)/G]$. If $X$ is a space with a $G$-action, let us write $e_c^+(X)$ resp.\ $e_c^-(X)$ for the Euler characteristic of the $(+1)$- resp.\ $(-1)$-eigenspaces of the $G$-action on the compact support cohomology. Equivalently,
		$$ e_c^+(X) = e_c([X/G]) \qquad \text{and} \qquad e_c^-(X) = e_c(X)-e_c^+(X).$$
		The group $G$ acts on the spectral sequence \eqref{SS}, and on the $E_2$-page it acts separately on both tensor factors. Since $e_c([T(\tau)/G])$ is the Euler characteristic of the $G$-invariants on the $E_2$-page, we find that 
		\[ e_c(\gamma(\tau)) = e_c([T(\tau)/G]) = e_c^+ (X_{2,s+1}) e_c^+(\mathbb G_m^k) + e_c^-(X_{2,s+1}) e_c^-(\mathbb G_m^k).\]
The group $G$ acts on $H^\bullet_c(\mathbb G_m^k)$ as the identity in even degrees, and by multiplication by $-1$ in odd degrees, which makes it easy to read off $e_c^+(\mathbb G_m^k)$ and $e_c^-(\mathbb G_m^k)$. In studying the $G$-action on $H^\bullet_c(X_{2,s+1})$ we again use the Leray spectral sequence of the fibration over $A_2$, since $G$ acts fiberwise. The group $G$ will act nontrivially on one of the tensor factors of $R\pi^{s+1}_\ast\Q_\ell \cong (R\pi^1_*\Qell)^{\otimes s+1}$; on this tensor factor, it will act on $R^i \pi_\ast\Q_\ell$ trivially for even $i$ and by $-1$ for odd $i$. We can then write 
$$ g_s(L) = g_s^+(L) + g_s^-(L)$$
with the $i$th coefficient of $g_s^+(L)$ (resp.\ $g_s^-(L)$) recording the multiplicity of $\V_{7,7}(-i)$ in the $(+1)$-eigenspace (resp.\ $(-1)$-eigenspace) of $R^{14+2i} \pi_*^{s+1} \Q_\ell$. In the same way that one derives \eqref{computation of gs} one finds that 
\[ g_6^+(L)= 1, \qquad g_6^-(L)=0,\qquad g_7^+(L) = 29+50L+29L^2, \qquad g_7^-(L)=7+14L+7L^2.\]
Then $e_c^\pm(X_{2,s+1})\equiv -g_s^\pm(L)S_{\ell}\langle 18\rangle$, and the result follows by putting the above together.
\end{proof}

\subsection{Non-cancellation}

\begin{lem}
    Let $s \in \{6,7\}$. Then $H^\bullet(\Xbar_{3,s})$ is not of Tate type.
\end{lem}

\begin{proof}We compute $e_c(\Xbar_{3,s})$ by adding up $e_c(\gamma(\tau))$ for every cone $\tau$. 

Consider first the case $s=6$. By \cref{lem:X36X37}, the contribution from the interior is $$e_c(X_{3,6}) = L^6 S_\ell\langle 18\rangle + 21 L^5 S_\ell\langle 18\rangle  + \dots $$ If $\tau$ is of rank $> 1$, then $e_c(\gamma(\tau))$ is a polynomial in $L$. We consider the coefficients of $L^6 S_\ell\langle 18\rangle$ and $L^5 S_\ell\langle 18\rangle$ in $e_c(\gamma(\tau))$ for $\tau$ of rank $1$. By \cref{lem: ec of torus rank one strata}, the coefficient is always nonpositive. By \cref{prop:number of torus rank one strata}, there are exactly $64$ cones $\tau$ with stabilizer group $G(\tau)\cong \Z/2$, and for each of them $e_c(\gamma(\tau))$ contains either a term $-L^6 S_\ell\langle 18\rangle$ or $-L^5 S_\ell\langle 18\rangle$ (depending on $\dim \tau$). Either way, adding up all of them we see that either the coefficient of $L^6 S_\ell\langle 18\rangle$ or $L^5 S_\ell\langle 18\rangle$ in $e_c(\Xbar_{3,6})$ is strictly negative. 

Now consider $s=7$. The contribution from the interior is $e_c(X_{3,7})=28 L^9 S_\ell\langle 18\rangle + \dots$ and again $e_c(\gamma(\tau))$ is a polynomial in $L$ if $\tau$ has rank $>1$. When $\tau$ has rank $1$, then by \cref{lem: ec of torus rank one strata}, the coefficient of $L^9 S_\ell\langle 18\rangle$ in $e_c(\gamma(\tau))$ is nonzero precisely for $\dim \tau^\aff =0$, in which case the coefficient is either $-36$ or $-29$. Either way, we see that the coefficient of $L^9 S_\ell\langle 18\rangle$ in $e_c(\Xbar_{3,7})$ must be strictly negative.
\end{proof}

This finishes the proof of \cref{thm:tateornot}.

\begin{rem}We point out a nontrivial consistency check of our computations. When computing $e_c(\Xbar_{3,s})$ for $s=6,7$ we found that the odd weight terms were all of the form $L^i S_\ell\langle 18\rangle$ for $i \leq 6$, resp.\ $i \leq 9$. In particular, it follows that $H^{31}(\Xbar_{3,6})=0$ and $H^{37}(\Xbar_{3,7})=0$. By Poincar\'e duality, this implies that $H^{17}(\Xbar_{3,6})=H^{17}(\Xbar_{3,7})=0$. Hence in both cases, the coefficient of $S_\ell\langle 18\rangle$ must cancel when adding up $e_c(\gamma(\tau))$ over all cones $\tau$. Let us directly verify this fact.

Consider first the case $s=6$. The coefficient of $S_\ell\langle 18\rangle$ in $e_c(\gamma(\tau))$ is $32$ when $\gamma(\tau)=X_{3,6}$ (\cref{lem:X36X37}), it is zero when $\tau$ has rank $>1$, and it is $(-1)^d$ if $\tau$ has rank $1$ and dimension $d$, except if $\dim \tau = 2$ and $G(\tau) \cong \Z/2$, in which case the coefficient is zero (\cref{lem: ec of torus rank one strata}).

Let $a_i$ (resp.~$b_i$) denote the number of rank $1$ cones of dimension $i+1$ whose stabilizer is trivial, resp.~cyclic of order $2$. Recall that $a_i=0$ for $i>1$. We obtain the equation
\begin{equation}\label{EC 1}
	a_0 + \sum_i (-1)^i b_i = 32. 
\end{equation} 
We can also directly prove \eqref{EC 1} as follows.
We obtain a cell decomposition of $T=(\R/\Z)^6$ with $a_i+2b_i$ cells in dimension $i$, and so we have $0=e(T)=\sum_i (-1)^i (a_i+2b_i)$.
Taking the half-sum of this equation with $a_0+a_1=2^s=64$ (Proposition \ref{prop:number of torus rank one strata}) recovers \eqref{EC 1}.

We now consider the case $s=7$. The coefficient of $S_\ell\langle 18\rangle$ in $e_c(\gamma(\tau))$ is $1408 = 11\cdot 128$ when $\gamma(\tau)=X_{3,7}$, and it is zero when $\tau$ has rank $>1$. It is $36(-1)^{d+1}$ if $\tau$ has rank $1$, dimension $d$, and trivial stabilizer. If $G(\tau)\cong \Z/2$ then the coefficient is $7$ or $-29$, when $\dim \tau$ is $1$ or $2$. Keeping notation as above, we find the equation
\begin{equation}\label{EC 2}
	-7 a_0 + 29 a_1 - 36 \sum_i (-1)^i b_i = 11 \cdot 128.
\end{equation}
Again this can be recovered as $-18$ times the equation $\sum_i (-1)^i (a_i+2b_i)=0$ plus $11$ times the equation $a_0+a_1=128$.

We also note that the vanishing of the $S_{\ell}\langle 18\rangle$-isotypic part of $H^{17}(\Xbar_{3,6})$ and $H^{17}(\Xbar_{3,7})$ follows from the proof of Theorem \ref{thm:holomorphic} in Section \ref{sec:holforms}.
\end{rem}

\section{Hodge numbers}\label{sec:Hodge}

\begin{proof}[Proof of Theorem \ref{thm:hodge}]
The result follows immediately from the Hodge--Tate comparison isomorphism \cite{faltings_padichodge} and Corollary \ref{cor:purepart} if we can find a prime $\ell$ (and $\iota:\C\simeq \Qellbar$) such that all Galois representations associated to automorphic representations in Theorem \ref{thm IH} are absolutely irreducible (conjecturally, they are always irreducible).
For symmetric powers of $2$-dimensional representations associated to level one (or more generally, non-CM of weight $\geq 2$) eigenforms, this is known to hold for all $\ell$ and $\iota$.
For the $4$-dimensional Galois representations associated to $\Delta_{w_1,w_2}$ ($23 \geq w_1>w_2>0$ odd) this is known for all but finitely many $\ell$ as a special case of \cite[Theorem B]{Ramakrishnan_irreducibility}.
\end{proof}

\subsection{Holomorphic forms}\label{sec:holforms}
Theorem \ref{thm:hodge} gives general constraints on the Hodge structure of $H^{k}(\Xbar_{g,s})$. Here, we give further strong constraints on the holomorphic (and thus antiholomorphic, by Hodge symmetry) part $H^{k,0}(\Xbar_{g,s})$, proving Theorem \ref{thm:holomorphic}. We also use holomorphic forms to prove Theorem \ref{thm:interiornontate}.

\begin{proof}[Proof of Theorem \ref{thm:holomorphic}]
We apply Proposition \ref{prop:allsubquotients} to reduce to the study of the groups $\IH^p(A_{g-r}^{\Sat}, \V_{\lambda}(-m))$.
Only the groups with $m=0$ can contribute nontrivial holomorphic forms.\footnote{We are using that $\IH^\bullet(A_g^\Sat, \V_\lambda)$ is effective, i.e.\ $h^{p,q}>0$ only if $p,q \geq 0$. Indeed, one can show it is effective by embedding it in the cohomology of a nonsingular compactification $\Xbar_{g,s}$.} 
We have $m\geq \binom{r+1}{2}+rs$, and hence $m\geq 1$ if $r\neq 0$. Therefore, we only need to consider $\IH^p(A_{g}^{\Sat}, \V_{\lambda})$ with $p+|\lambda|=k$ and $\lambda_1\leq s$. The cases where the space of holomorphic forms vanishes then follow by examining Tables \ref{tab:C} and \ref{tab:DEF}. For the non-vanishing cases, we argue as in the proof of Theorem \ref{thm:tateornot}: there is an injection of Galois representations  \[\IH^{\bullet-q}(A^{\Sat}_g, R^q \pi_* \Qell) \hookrightarrow H^\bullet(\Xbar_{g,s},\Qell).\]
When $g=3$ and $s\geq 8$ or $4\leq g \leq 7$ and $s\geq c(g)$, we find the Galois representation $\Sym^2 S_{\ell}\langle 11 \rangle$, coming from the first five rows of Table \ref{tab:DEF}. The associated Hodge structure $\mathrm{Hdg}(\Sym^2 S_{\ell}\langle 12 \rangle)$ has a nonzero space of holomorphic $22$-forms. The result follows by the Hodge--Tate comparison.\end{proof}

\begin{rem}\label{rem:g12}\cref{thm:holomorphic} only treats the case $g \geq 3$. The cases $g\in \{1,2\}$ work out somewhat differently. In these cases, explicit formulas for $\dim H^{k,0}(\Xbar_{g,s})$, for all values of $k$, can be extracted from known results in the literature. In genus one, $\Xbar_{1,s}$ is birational to $\Mbar_{1,s-1}$, the moduli space of stable curves of genus $1$ with $s-1$ markings. Because the spaces of holomorphic forms are birational invariants for nonsingular spaces, the formula in \cite[Proposition 3.1]{CLPWFA} gives the answer.

In genus two, the Abel--Jacobi map 
    \[
    M_{2,2s}/\mathfrak{S}_2^s\rightarrow X_{2,s} 
    \]
    given by $(C,p_1,\dots,p_{2s})\mapsto(\mathrm{Jac}(C),\omega_C^*(p_1+p_2),\dots, \omega_C^*(p_{2s-1}+p_{2s}))$ is birational, so any toroidal compactification $\Xbar_{2,s}$ is birational to $\Mbar_{2,2s}/\mathfrak{S}_2^s$. Therefore, it suffices to compute $H^{k,0}(\Mbar_{2,2s})^{\mathfrak{S}_2^s}$. Now \cite[Proposition 3.5]{CLPWFA} computes $H^{k,0}(\Mbar_{2,2s})$ as a representation of the symmetric group $\mathfrak S_{2s}$ in terms of the $(k,0)$ part of $W_k H^{3}(A_2,\V_{\lambda})$ for $|\lambda|=k-3$; the latter are expressed in terms of spaces of vector-valued Siegel modular forms in \cite{localsystemsA2} (or, for that matter, by the results of this paper). This allows one to write down a not-too-compact formula for the answer.
\end{rem}

\begin{proof}[Proof of Theorem \ref{thm:interiornontate}]
    Let $\Xbar_{g,s}$ be any nonsingular toroidal compactification of $X_{g,s}$. We have a short exact sequence
    \[
    H^{k-2}(\tilde{\partial X}_{g,s},\C)\rightarrow H^k(\Xbar_{g,s},\C)\rightarrow W_kH^k(X_{g,s},\C)\rightarrow 0,
    \]
    where $\tilde{\partial X}_{g,s}$ is the normalization of the boundary of $X_g$ in $\Xbar_g$. The first morphism is of Hodge type $(1,1)$. Hence, the second map restricted to the space of holomorphic $k$-forms $H^{k,0}(\Xbar_g)\rightarrow F^k W_k H^k(X_{g,s},\C)$ is an isomorphism. Therefore, if $H^{k,0}(\Xbar_g)\neq 0$ for some $k>0$, then $H^\bullet(X_{g,s})$ is not Tate. 
    For $g=1,2$ and $s\geq c(g)$, it is well-known that $\Xbar_{g,s}$ admits non-trivial holomorphic forms (see Remark \ref{rem:g12}).  When $g=3$, we use Lemma \ref{lem:X36X37}, which shows that $H^\bullet(X_{3,6})$ is not Tate. Because $X_{3,s}\rightarrow X_{3,6}$ is proper for $s\geq 6$, it follows that $H^{\bullet}(X_{3,s})$ is not Tate for any $s\geq 6$.
    For $4\leq g\leq 7$ and $s\geq c(g)$, there exist holomorphic $22$-forms by Theorem \ref{thm:holomorphic}.

    For $g\geq 9$, we argue as in the proof of Theorem \ref{thm:tateornot} and use the Hodge--Tate comparison. The parameters $\Delta^{(j)}_{2g-3}[4]\oplus [2g-7]$ contribute the Galois representations $\bigoplus_j\Sym^2 \rho_{\Delta^{(j)}_{2g-3}}$. The associated Hodge structures $\mathrm{Hdg}(\Sym^2 \rho_{\Delta^{(j)}_{2g-3}})$ contribute holomorphic $4g-6$ forms. By the decomposition theorem and the comparison isomorphism, it follows that $H^{4g-6,0}(\Abar_g)\neq 0$ for $g=7$ or $g\geq 9$. 

    For $g=8$, the above argument fails because there are no cusp forms of weight $14$ for $\SL_2(\Z)$, and among the parameters $\psi$ contributing to $\IH^\bullet(A_8)$, only $\Delta_{11}[6] \oplus [5]$ has non-Tate contribution which occurs with a negative Tate twist so $\IH^\bullet(A_8)$ only has positive Hodge-Tate weights.
    Instead, we check that the Euler characteristic $e_c(A_8)$ is not Tate as in \cite[\S 9.2]{taibi_ecAn}.
    We have an explicit formula of the form
    \[ e_c(A_8) = \sum_{g', \lambda', m'} n(\lambda',m') e_\IH(A_{g'}^\Sat, \V_{\lambda'}(-m')) \]
    where $g' \leq g$, $\lambda'+g' \leq 8$ and $m' \geq 0$ and $n(\lambda',m') \in \Z$.
    The parameters contributing to $\IH^\bullet(A_{g'}^\Sat, \V_{\lambda'})$ are of the form $\psi = \oplus_i \psi_i$ with
    \[ \psi_i \in \{ [2d+1], \Delta_{11}[2], \Delta_{15}[2], \Delta_{11}[4], \Delta_{11}[6] \} \]
    and only those containing $\Delta_{11}[4]$ or $\Delta_{11}[6]$ have non-Tate contributions.
    These are either of the form $\Delta_{11}[4] \oplus [2d+1]$ or $\Delta_{11}[6] \oplus [2d+1]$, by regularity.
    The latter do not contribute non-Tate classes in weight less than $28$, while the former contribute $\Sym^2 S_\ell \langle 12 \rangle$ to $\IH^{22}(A_{d+4}^\Sat, \V_{\lambda'})$ modulo Tate classes (see \ref{it:typeE}).
    Considering only weight $22$ to simplify the computation, so that we need only consider terms corresponding to $m'=0$ and $\lambda'=(7-g',7-g',7-g',7-g',0,\dots,0)$ in the above sum, we find that the weight $22$ part of $e_c(A_8)$ is equal to $-\Sym^2 S_\ell \langle 12 \rangle$ (modulo Tate classes).

    Now we prove the converse. We have that $\Xbar_{g,s}$ is a normal crossings compactification of $X_{g,s}$ with boundary strata corresponding to cones $\tau$. We consider the weight spectral sequence \cite[Example 3.5]{spectralsequencestratification} computing the compactly supported cohomology $H_c^{\bullet}(X_{g,s})$. The entries on the first page of the spectral sequence are the cohomology groups of the normalizations of the closures of the locally closed strata $\gamma(\tau)$ described in Section \ref{sec:compact}. We will prove that all of the entries on the first page of the spectral sequence are Tate. 
    
    The normalizations of the strata closures are smooth and proper, and each is stratified by $\gamma(\tau')$, where $\tau'$ is a face of $\tau$. Using the spectral sequence \eqref{BMSS}, it suffices to show that $\gr^W_i H^i_c(\gamma(\tau))$ is Tate for all $i$ and all $\tau$ in the cone decomposition defining $\Xbar_{g,s}$. By Lemma \ref{stratification lemma 2}, $\gr^W_i H^i_c(\gamma(\tau))$ is a direct sum of subquotients of Galois representations of the form $\gr^W_{i} H_c^p(A_{g-r},\V_{\lambda}(-m))$, where $(\lambda,m)$ ranges over dominant weights of $\mathrm{GSp}_{2(g-r)}$ satisfying $p+|\lambda|+2m=i$, $m\geq \binom{r+1}{2}+rs-\dim(\tau)$, and $\lambda_1 \leq s+r$, where $r=\mathrm{rank}(\tau)$. Each $\gr^W_{i} H_c^p(A_{g-r},\V_{\lambda}(-m))$ embeds in the corresponding intersection cohomology group $\IH^p(A^{\Sat}_{g-r},\V_{\lambda}(-m))$. First, if $i\leq 21=\dim A_6$, then we see that the contributions are all Tate by examining Tables \ref{tab:rest}, \ref{tab:C}, and \ref{tab:DEF}, using the bound $0\leq r \leq 6$. If $i>21$, we apply Poincar\'e duality to the intersection cohomology groups $\IH^p(A_{g-r}^{\Sat},\V_{\lambda}(-m))$ to reduce to the case $i\leq 21$.\end{proof}

\section{Relative Lie algebra cohomology and \texorpdfstring{$L^2$}{L2}-cohomology}\label{sec:L2}

In the proofs of our main results we have used that $\IH^\bullet(A_g^\Sat,\V_\lambda)$ decomposes into summands indexed by Arthur--Langlands parameters. This allows us to calculate $\IH^\bullet(A_g^\Sat,\V_\lambda)$ (an algebro-geometric object) in terms of data involving automorphic representations. There is however also a more direct connection between automorphic representations and $\IH^\bullet(A_g^\Sat,\V_\lambda)$, via $(\mathfrak g,K)$-cohomology and the \emph{Matsushima--Murakami formula}. It seems plausible that some version of our main results could be derived using this connection, too, but we have not attempted to carry this out. Moreover, an approach using $(\mathfrak g,K)$-cohomology would only see the ``shadow'' of our results involving Hodge numbers and real Hodge structures, as opposed to our results in $\ell$-adic cohomology, where we explicitly identify the Galois representations in cohomology. In any case, let us briefly expound upon this connection.

Let us temporarily denote by $\V_\lambda$ the real polarized variation of Hodge structure on $A_g$ associated to the irreducible representation $V_\lambda$ of $\Sp_{2g}$. Inside the real de Rham complex of $\V_\lambda$, one may consider the subcomplex consisting of differential forms $\omega$ such that $\omega$ and $d\omega$ are both square integrable. Its cohomology is by definition the $L^2$-cohomology of $A_g$, and we denote it $H^\bullet_{(2)}(A_g,\V_\lambda)$. 
The \emph{Zucker conjecture}, proven independently by Looijenga \cite{looijengazucker} and Saper--Stern \cite{saperstern}, expresses the transcendentally defined $L^2$-cohomology groups in terms of intersection cohomology:
\[ H^\bullet_{(2)}(A_g,\V_\lambda) \cong \IH^\bullet(A_g^\Sat,\V_\lambda).\]
This holds more generally for any Shimura variety. Both sides of the isomorphism are naturally equipped with real Hodge structures, using the representation of $L^2$-cohomology by harmonic forms on the left-hand side, and Saito's theory of mixed Hodge modules on the right-hand side, and it is a recent result of Looijenga \cite{Looijenga2025} that it is an isomorphism of real Hodge structures.

Recall from \cite{borel-regularization} and \cite{borelcasselman} that the $L^2$-cohomology groups $H^\bullet_{(2)}(A_g, \V_\lambda)$ may be computed from the level one discrete automorphic spectrum of $G=\GSp_{2g}$.
For simplicity we complexify $\V_\lambda$ in the following discussion, i.e.\ we assume that $V_\lambda$ is an irreducible algebraic representation of $\Sp_{2g,\C}$, extended as usual to a representation of $G_\C$.
This does not incur a loss of information as the isomorphism class of a representation of $\mathbb{S} := \Res_{\C/\R} \GL_{1,\C}$ (i.e.\ a real Hodge structure) is determined by its base change to $\C$.
Denote \(G = \GSp_{2g}\), specifically \(G(A) = \{ (x,c) \in \GL_{2g}(A) \times A^\times \,|\, {}^t x J_g x = c J_g\}\) where
\[ J_g = \begin{pmatrix} 0 & I_g \\ -I_g & 0 \end{pmatrix}. \]
Realize \(A_g\) (over \(\C\)) as the Shimura variety in level \(G(\Zhat)\) associated to the Shimura datum \((G,X)\) where \(X\) is the \(G(\R)\)-orbit of
\begin{align*}
  h : \mathbb{S} & \longrightarrow G_{\R} \\
  a + i b & \longmapsto \begin{pmatrix} aI_g & -bI_g \\ bI_g & aI_g \end{pmatrix}.
\end{align*}
Let \(K \simeq \R_{>0} \times U(g)\) be the centralizer of \(h\) in \(G(\R)\).
The morphism \(h\) may be uniquely written on \(\mathbb{S}(\R) = \C^\times\) as \(z \mapsto \mu_h(z) \mu_h'(\ol{z})\) where \(\mu_h, \mu_h': \GL_{1,\C} \to G_\C\).
Let $\mathcal{A}^2(G, \omega_\lambda^{-1})$ be the space of automorphic forms $G(\Q) \backslash G(\A) \to \C$ which transform under $\R_{>0} \subset G(\R)$ by the inverse of the central character $\omega_\lambda$ of $V_\lambda$ and which after a suitable twist are square-integrable on $G(\Q) \backslash G(\A) / \R_{>0}$.
We have an identification
\[ H^\bullet_{(2)}(A_g, \V_\lambda) \simeq H^\bullet(\gfrak, K; \mathcal{A}^2(G, \omega_\lambda^{-1})^{G(\Zhat)} \otimes V_\lambda) \]
where the right-hand side is defined as the cohomology of the Chevalley-Eilenberg complex (differentials omitted)
\begin{equation} \label{eq:def_g_K_coh}
  \Hom_K(\wedge^\bullet \gfrak/\kfrak, \mathcal{A}^2(G, \omega_\lambda^{-1})^{G(\Zhat)} \otimes V_\lambda).
\end{equation}
Let us recall the complex Hodge structure on these cohomology groups, as a representation of \(\mathbb{S}_{\C}\) (according to the convention in \cite[(1.1.1.1)]{deligne-shimuravarieties}).
Let \(\pfrak^{\pm} \subset \gfrak\) be the eigenspace on which \(\operatorname{Ad} \mu_h(z)\) acts by \(z^{\pm 1}\), so that we have a decomposition \(\gfrak = \kfrak \oplus \pfrak\) with \(\pfrak = \pfrak^+ \oplus \pfrak^-\), in particular
\[ \wedge^\bullet \gfrak/\kfrak \simeq \wedge^\bullet \pfrak^+ \oplus \wedge^\bullet \pfrak^-. \]
We let \(z \in \mathbb{S}(\R)\) act by \(z\) on \(\gfrak^+\) and by \(\ol{z}\) on \(\gfrak^-\) (note that this is not the adjoint action composed with \(h\)).
We endow \(V_\lambda\) with the action of \(\mathbb{S}_\C\) obtained by composing the action of \(G_\C\) with \(h_\C\).
Finally we endow \(\mathcal{A}^2(G, \omega_\lambda^{-1})\) with the trivial action of \(\mathbb{S}_\C\), giving us an action of \(\mathbb{S}_\C\) on \eqref{eq:def_g_K_coh}.

In general making this description more precise would require detailed knowledge of (part of) the discrete automorphic spectrum of $G$, but fortunately in level one the restriction map
\[ \mathcal{A}^2(G, \omega_\lambda^{-1})^{G(\Zhat)} \longrightarrow \mathcal{A}^2(\Sp_{2g})^{\Sp_{2g}(\Zhat)} \]
is bijective, and the above $(\gfrak,K)$-cohomology groups are obviously just $(\mathfrak{sp}_{2g},U(g))$-cohomology groups.
Relying on Arthur's endoscopic classification \cite{arthurclassification} for $\Sp_{2g,\Q}$, the identification of cohomological Arthur-Langlands packets for $\Sp_{2g}(\R)$ with the (more concrete) Adams-Johnson packets \cite{AMR} and the explicit computation of $(\mathfrak{sp}_{2g},U(g))$-cohomology of irreducible unitary $(\mathfrak{sp}_{2g},U(g))$-modules \cite[Proposition 6.19]{vz}, one can in principle explicitly compute $H^\bullet_{(2)}(A_g, \V_\lambda)$ explicitly as a representation of $\mathbb{S}_\C$ (even adding the commuting action of the Hecke algebra $\mathcal{H}(\Sp_{2g}(\A_f), \Sp_{2g}(\Zhat))$).
But Arthur gave in \cite[\S 9]{arthur_unip} (relying on \cite{vz}) a more conceptual description of the $L^2$-cohomology of a Shimura variety, conditional on the endoscopic classification of the discrete automorphic spectrum of the group underlying the Shimura datum.
While this classification is not known in full generality for $\GSp_{2g}$, as we saw above in level one we can reduce the matter to the endoscopic classification for $\Sp_{2g}$.
Simply paraphrasing \cite[Proposition 9.1]{arthur_unip} (using Deligne's sign convention) yields Theorem \ref{thm:Hodge_g_K_A_g} below, for which we need to recall a few notions.

As usual denote \(\tau = \lambda + \rho\), and consider \(\psi = \psi_0 \oplus \dots \oplus \psi_r \in \tPsidut(\Sp_{2g})\).
For each \(0 \leq i \leq r\) we have a local Arthur-Langlands parameter \(\psi_{i,\infty}: W_\R \times \SL_2(\C) \to \mathcal{M}_{\psi_i}(\C) \simeq \SO_{n_i d_i}(\C)\) such that composing with the standard representation of \(\mathcal{M}_{\psi_i}\) yields \(\operatorname{LL}(\pi_{i,\infty}) \otimes \Sym^{d_i-1} \Std_{\SL_2}\) where \(\operatorname{LL}(\pi_{i,\infty}): W_\R \to \GL_{n_i}(\C)\) is the Langlands parameter of \(\pi_{i,\infty}\).
To be concrete the Langlands parameter \(\operatorname{LL}(\pi_{i,\infty})\) is characterized as follows.
\begin{itemize}
\item For \(i>0\) it is the unique representation of \(W_\R\) whose restriction to \(\C^\times\) is equivalent to
  \[ \bigoplus_{j=1}^{n_i/2} z \longmapsto \operatorname{diag}((z/|z|)^{2w_j(\pi_i)}, (z/|z|)^{-2w_j(\pi_i)}). \]
  \item For \(i=0\) it is the unique representation of \(W_\R\) with trivial determinant whose restriction to \(\C^\times\) is equivalent to
  \[ 1 \oplus \bigoplus_{j=1}^{(n_0-1)/2} z \longmapsto \operatorname{diag}((z/|z|)^{2w_j(\pi_i)}, (z/|z|)^{-2w_j(\pi_i)}). \]
\end{itemize}
Let \(\mathcal{M}_{\psi_i,\mathrm{sc}}\) be the simply connected cover of \(\mathcal{M}_{\psi_i}\) and let \(\mathcal{GM}_{\psi_i,\mathrm{sc}} = (\GL_1 \times \mathcal{M}_{\psi_i,\mathrm{sc}}) / \mu_2\) where the subgroup \(\mu_2\) is diagonally embedded (in the second factor its image is \(\ker \mathcal{M}_{\psi_i,\mathrm{sc}} \to \mathcal{M}_{\psi_i}\)).
Note that \(\mathcal{GM}_{\psi_i,\mathrm{sc}}\) is isomorphic to \(\GSpin_{n_i d_i}\).
We have lifts \(\tilde{\psi}_{i,\infty}: W_\R \times \SL_2(\C) \to \mathcal{GM}_{\psi_i,\mathrm{sc}}(\C)\) such that the composition with \(\GSpin_{n_i d_i}(\C) \to \C^\times\), $(t,g) \mapsto t^2$ is given on \(\C^\times \subset W_\R\) by \(z \mapsto (z \ol{z})^{|\tau_i|}\) (up to a twist this is \cite[Proposition-Definition 3.4.4]{taibi_ecAn}).
These lifts may not be unique but their restrictions to \(\C^\times \times \SL_2(\C)\) are unique.
Let \(\iota_\R: \C^\times \to W_\R \times \SL_2(\C)\), \(z \mapsto (z, \operatorname{diag}((z \ol{z})^{1/2}, (z \ol{z})^{-1/2}))\).
Let \(\spin_{\psi_0}\) be the spin representation of \(\mathcal{GM}_{\psi_0,\mathrm{sc}}\), and recall from \cite[Definition 4.7.4]{taibi_ecAn} the half-spin representations \(\spin_{\psi_i}^\pm\) of \(\mathcal{GM}_{\psi_i,\mathrm{sc}}\) for \(i>0\).
To be precise for $i>0$ the half-spin representation $\spin_{\psi_i}^+$ is characterized among $\spin_{\psi_i}^\pm$ by the property that the representation $\spin_{\psi_i}^+ \circ \tilde{\psi}_{i,\infty} \circ \iota_\R$ contains a character $z \mapsto z^a \ol{z}^b$ with $a = |\tau_i|$.
The signs $(u_i(\psi))_{1 \leq i \leq r}$ were recalled before Theorem \ref{thm IH tables}.

\begin{thm} \label{thm:Hodge_g_K_A_g}
  The representation \eqref{eq:def_g_K_coh} of \(\mathbb{S}(\R)\) decomposes as \(\bigoplus_\psi H_\psi\), ranging over \(\psi \in \tPsidut(\Sp_{2g})\), with \(H_\psi\) isomorphic to the composition with \(\iota_\R\) of the \emph{dual} of the representation
  \[ \spin_{\psi_0} \circ \tilde{\psi}_{0,\infty} \otimes \bigotimes_{i=1}^r \spin_{\psi_i}^{u_i(\psi)} \circ \tilde{\psi}_{i,\infty} \]
  of \(W_\R \times \SL_2(\C)\).
\end{thm}
\begin{proof}
  As explained above this formula is proved by paraphrasing the proof of \cite[Proposition 9.1]{arthur_unip} using Arthur's endoscopic classification for $\Sp_{2g}$ and \cite{AMR}, decomposing final representations into tensor products of (half-)spin representations exactly as in \cite[\S 4.7]{taibi_ecAn}.
  Note that we adopted Deligne's sign conventions so one has to replace the minuscule representation $r_\mu$ in \cite[\S 9]{arthur_unip} by its dual $r_{-\mu}$.
\end{proof}

\begin{rem} \label{rem:comp_Arthur_Kottwitz}
  Of course this is very similar to \cite[Theorem 4.7.2]{taibi_ecAn}, and indeed \cite[Proposition 9.1]{arthur_unip} is the Hodge analogue of Kottwitz' conjecture \cite{kottwitz} on $\ell$-adic \'etale intersection cohomology of minimal compactifications of Shimura varieties.
  To our knowledge, contrary to the case of compact Shimura varieties, there is no de Rham (or even just Hodge-Tate) $\ell$-adic comparison isomorphism relating the two.
  For this reason we do not yet know the Hodge-Tate weights of the Galois representations $\sigma_{\psi_i,\iota}^{\spin,u_i(\psi)}$ occurring in \cite[Theorem 7.1.3]{taibi_ecAn} (when $i>0$), except for those having dimension $2$ (when $\psi_i$ has dimension $4$) and $8$ (when $\psi_i$ has dimension $8$).
\end{rem}

\bibliographystyle{alpha}
\bibliography{database}

\end{document}